\documentclass[11pt]{amsart}   	
\usepackage{geometry}             
\usepackage[dvipsnames]{xcolor}   		             		
\usepackage{graphicx}				
\usepackage{amssymb,mathrsfs}
\usepackage{amssymb,amsthm,amsmath,stmaryrd}
\SetSymbolFont{stmry}{bold}{U}{stmry}{m}{n}
\usepackage{tikz}
\usepackage{tikz-cd}
\usepackage{accents,upgreek,enumerate}
\usepackage[headings]{fullpage}
\usepackage{bm}
\usepackage[all]{xy}
\usepackage{caption}

\newtheorem{innercustomthm}{Theorem}
\newenvironment{customthm}[1]
  {\renewcommand\theinnercustomthm{#1}\innercustomthm}
  {\endinnercustomthm}
  
\usetikzlibrary{calc}
\usetikzlibrary{fadings}
\usetikzlibrary{decorations.pathmorphing}
\usetikzlibrary{decorations.pathreplacing}

\newcounter{marginnote}
\def\overnorm#1{\overline{#1}\vphantom{#1}}

\DeclareMathAlphabet{\mathpzc}{OT1}{pzc}{m}{it}

\usepackage[backref]{hyperref}
\hypersetup{
  colorlinks   = true,          
  urlcolor     = violet,          
  linkcolor    = blue,          
  citecolor   = purple             
}

\newtheorem{theorem}{Theorem}[subsection]
\newtheorem{corollary}[theorem]{Corollary}
\newtheorem{lemma}[theorem]{Lemma}
\newtheorem{proposition}[theorem]{Proposition}

\newtheorem{quasi-theorem}[theorem]{Quasi-Theorem}

\theoremstyle{definition}
\newtheorem{definition}[theorem]{Definition}
\newtheorem{warning}[theorem]{Warning}

\newtheorem{question}[theorem]{Question}

\newtheorem{blank remark}[theorem]{}

\theoremstyle{remark}
\newtheorem{rem1}[theorem]{Remark}
\newenvironment{remark}{\begin{rem1}\em}{\end{rem1}}

\newtheorem{not1}[theorem]{Notation}

\newcommand{\CC} {{\mathbf C}}          
            
\newcommand{\NN} {{\mathbf N}}		
\newcommand{\PP}{\mathbf{P}}         
		
\newcommand{\RR} {{\mathbf R}}		
\newcommand{\ZZ} {{\mathbf Z}}	
\newcommand{\GG} {{\mathbf G}}		

\def\Gm{\mathbf{G}_m}

\def\expdim{\operatorname{exp.dim}}

\newcommand{\Hom}{\operatorname{Hom}}

\DeclareMathOperator{\val}{val}
 
\DeclareMathOperator{\res}{res}

\DeclareMathOperator{\spec}{Spec}

\newcommand{\cal}{\mathcal}

\def\cM{{\cal M}}

\def\fM{\mathfrak{M}}

\newcommand{\plC}{\kern1pt\raisebox{-.5pt}{\scalebox{0.8}[1.3]{$\sqsubset$}}}

\newcommand{\Mbar}{\overnorm{\cM}\vphantom{\cM}}

\def\trop{\mathrm{trop}}
\def\an{\mathrm{an}}

\newcommand{\Spec}{\operatorname{Spec}}

\makeatletter
\def\blfootnote{\xdef\@thefnmark{}\@footnotetext}
\makeatother

\title[Moduli of stable maps in genus one: logarithmic theory]{{\larger M}oduli of stable maps in genus one {\it \&} logarithmic geometry II}

\date{}

\author[Ranganathan]{Dhruv Ranganathan}
\address{Dhruv Ranganathan \\ Department of Pure Mathematics and Mathematical Statistics\\ University of Cambridge}
\email{\href{mailto:dr508@cam.ac.uk}{dr508@cam.ac.uk}}

\author[Santos-Parker]{Keli Santos-Parker}
\address{Keli Santos-Parker \\ Medical School\\ University of Michigan}
\email{\href{mailto:parkerks@med.umich.edu}{parkerks@med.umich.edu}}

\author[Wise]{Jonathan Wise}
\address{Jonathan Wise \\ Department of Mathematics\\
University of Colorado}
\email{\href{mailto:jonathan.wise@colorado.edu}{jonathan.wise@colorado.edu}}

\begin{document}

\begin{abstract}
This is the second in a pair of papers developing a framework to apply logarithmic methods in the study of stable maps and singular curves of genus $1$. This volume focuses on logarithmic Gromov--Witten theory and tropical geometry. We construct a logarithmically nonsingular and proper moduli space of genus $1$ curves mapping to any toric variety. The space is a birational modification of the principal component of the Abramovich--Chen--Gross--Siebert space of logarithmic stable maps and produces logarithmic analogues of Vakil and Zinger's genus one reduced Gromov--Witten theory. We describe the non-archimedean analytic skeleton of this moduli space and, as a consequence, obtain a full resolution to the tropical realizability problem in genus $1$.
\end{abstract}

\maketitle

\vspace{-0.2in}

\setcounter{tocdepth}{1}
\tableofcontents

\section{Introduction}


This paper is the second in a pair, exploring the interplay between tropical geometry, logarithmic moduli theory, stable maps, and  moduli spaces of genus~$1$ curves. In the first volume, we used this interplay to construct new nonsingular moduli spaces compactifying the space of elliptic curves in projective space via \textbf{radially aligned} stable maps and quasimaps. In this paper, we focus on applications to logarithmic Gromov--Witten theory and tropical geometry.

\noindent
{\bf I. Realizability of tropical curves.} We give a complete characterization of genus $1$ tropical maps that can be realized as tropicalizations of genus $1$ curves mapping to tori, completing a study initiated in Speyer's thesis. We show that a combinatorial condition identified by Baker--Payne--Rabinoff is always sufficient. Our proof is independent of these previous results, and is based on the geometry of logarithmic maps.

\noindent
{\bf II. Logarithmic stable maps.} We construct a toroidal moduli space parametrizing maps from pointed genus $1$ curves to any toric variety with prescribed contact orders along the toric boundary. This is a desingularization of the principal component of the space of logarithmic stable maps. The boundary complex of this compactification is identified as a space of realizable tropical maps.

\subsection{Superabundant tropical geometries}
The realization problem is the crux of the relationship between tropical geometry to algebraic geometry, and is unavoidable in enumerative applications. Given an abstract tropical curve\footnote{The first author continues his efforts to popularize Dan Abramovich's convention that algebraic curves be denoted by $C$, $\mathscr C$, while tropical curves be denoted $\plC$, approximating their appearance in nature.} $\plC$ of genus $g$ and a balanced piecewise linear map
\[
F: \plC\to \RR^r,
\]
we ask, \textit{does there exist a non-archimedean field $K$ extending $\CC$, a smooth algebraic curve $C$ over $K$ and a map
\[
\varphi: C\to \mathbf G_m^r,
\]
such that $\varphi^{\trop}$ coincides with $F$?} 

{
When $\plC$ has genus $0$, the only obstruction to lifting is the local balancing condition, and all tropical curves satisfying that condition are realizable.  This is reflected in the logarithmic smoothness of the moduli space of genus $0$ logarithmic maps~\cite{NS06,R15b,Tyo12}. In genus $1$, nonlocal obstructions already appear
}%
for maps $\plC\to \RR$. The obstructions appear when the circuit of $\plC$ is contained in a proper affine subspace of $\RR^r$. Speyer discovered a sufficient condition for realizability in 2005~\cite{Sp-thesis,Sp07}.  A weaker necessary condition was identified in~\cite[Section 6]{BPR16}.  We provide a characterization of the realizable tropical curves in genus $1$ in Theorem~\ref{thm: tropical} in terms of the geometry of the skeleton of an analytic space of maps.

Let $\Gamma$ be a marked tropical curve of genus $1$ with a unique vertex and $n$ half-edges. Fix a balanced map $\Gamma\to \RR^r$. Let $\cM_{\Gamma}(\mathbf G_m^r)$ be the moduli space of maps 
\[
\varphi: C\to \mathbf G_m^r,
\]
where $C$ is a non-compact smooth algebraic curve of genus $1$ with $n$ punctures, and the vanishing orders at infinity of these punctures are specified by the slopes along the edges of $\Gamma$ in $\RR^r$. Let $W_\Gamma(\RR^r)$ be the corresponding set of tropical maps
\[
\plC\to \RR^r
\]
whose recession fan is given by $\Gamma\to \RR^r$, and satisfy the \textbf{well-spacedness condition}, as defined in Section~\ref{sec: tropical-moduli}.\footnote{We caution the reader that the meaning attributed by Speyer to well-spacedness is stronger than the one we use here; see Warning~\ref{warning: diff-defs}.} This set can be given the structure of a generalized cone complex.

Given a map $\varphi:\mathscr C\to \mathbf G_m^r$ over a valued field, one obtains a balanced piecewise linear map from a Berkovich skeleton $\plC$ of the punctured general fiber curve $\mathscr C_\eta$ to $\RR^r$, i.e., to the skeleton of the torus~\cite{R16}. This piecewise linear map is the \textbf{tropicalization} of $\varphi$ and is denoted $\varphi^{\trop}$. 

\begin{customthm}{A}\label{thm: tropical}
There exists a continuous and proper tropicalization map
\[
\trop: \cM_\Gamma^{\an}(\mathbf G_m^r)\to W_\Gamma(\RR^r)
\]
sending a map $[\varphi]$ over a valued field to its tropicalization. There is generalized cone complex $P_\Gamma(\RR^r)$ and a finite morphism
\[
\trop_{\mathfrak S}: P_\Gamma(\RR^r) \to W_\Gamma(\RR^r),
\]
which is an isomorphism upon restriction to each cone of the source. The degree of this finite morphism is explicitly computable and the complex $P_{\Gamma}(\RR^r)$ is a skeleton of the analytic moduli space $\cM_\Gamma^{\an}(\mathbf G_m^r)$.
\end{customthm}

{
The theorem is proved as Theorem~\ref{thm:nonarch} of the text.
}%

The statement that the tropicalization has a finite cover that is a skeleton is a toroidal version of the {sch\"on} condition, frequently cited in tropical geometry. The skeleton $P_\Gamma(\RR^r)$ functions as a {parametrizing complex} for the tropicalization, as in work of Helm and Katz~\cite{HK12,Tev07}. 



\subsection{Logarithmic stable maps} Our tropical investigation leads naturally to an understanding of the geometry of space of logarithmic stable maps to toric varieties in genus $1$. The open moduli problem we consider is that of maps
\[
(C,p_1,\ldots, p_m) \to Z,
\]
where $C$ is a smooth pointed curve of genus $1$, the target $Z$ is a toric variety, and the contact orders of the points $p_i$ with the boundary divisors on $Z$ is fixed. There is a natural modular compactification of this space via the theory of logarithmic stable maps, due to Abramovich--Chen and Gross--Siebert~\cite{AC11,Che10,GS13}. When the genus of the source curve is $0$, the resulting moduli space is logarithmically smooth, but in genus $1$ can be highly singular and non-equidimensional. We use the insights of Theorem~\ref{thm: tropical} to construct a logarithmically smooth modular compactification, in parallel with the desingularization of the ordinary stable maps space due to Vakil and Zinger~\cite{RSW17A,VZ08}.

Let $Z$ be a proper complex toric variety and $\mathscr L_\Gamma(Z)$ the moduli space of genus $1$ logarithmic stable maps to $Z$ with discrete data $\Gamma$, i.e., $\Gamma$ records the genus and the contact orders of the marked points with the toric boundary of $Z$. Let $\mathscr L_\Gamma^\circ(Z)$ be the locus of parametrizing maps {of positive degree} from smooth domains, and let $\overnorm{\mathscr L_\Gamma^\circ}(Z)$ be the closure.

\begin{customthm}{B}\label{thm: toric-targets}
Consider the following data as a moduli problem on logarithmic schemes:
\begin{enumerate}
\item a family of $n$-marked, radially aligned logarithmic curves $C\to S$,
\item a logarithmic stable map $f:C\to Z$ with contact order $\Gamma$,
\end{enumerate}
such that the map $f$ is \textbf{well-spaced} (see Definition~\ref{def:well-spaced}). This moduli problem is represented by a proper and logarithmically smooth stack with logarithmic structure $\mathcal W_\Gamma(Z)$ and the natural morphism
\[
\mathcal W_\Gamma(Z)\to \overnorm{\mathscr L_\Gamma^\circ}(Z)
\]
is proper and birational.
\end{customthm}

{See Theorem~\ref{thm: realizability} for the proof.}

The well-spacedness property above is efficiently stated in tropical language, and this is done later in the paper. At a first approximation it may be thought of as forcing a factorization property after composing $C\to Z$ with any rational map $Z\dashrightarrow \PP^1$ induced by a character. These logarithmic maps are precisely the ones that have \textbf{well-spaced} tropicalizations. A prototype for practical calculations on this space may be found in~\cite{LR15}.

\subsection{Motivation for the construction} The combinatorics of logarithmic stable maps are essentially part of tropical geometry. Indeed, if the variety $Z$ is a toric variety, taken with its toric boundary, the analytification of the moduli space of logarithmic maps maps continuously to a polyhedral complex parametrizing tropical curves~\cite{R16}. The connection is especially transparent in genus $0$, see~\cite{R15b}. In genus $1$, the tropical realizability problem can be used to predict the desingularization above, as we now explain. The moduli space $\mathscr L_\Gamma(Z)$ of genus $1$ logarithmic stable maps is highly singular, however, it maps naturally to a logarithmically smooth Artin stack. More precisely, if $\mathscr A_Z = [Z/T]$ is the Artin fan of $Z$ obtained by performing a stack quotient on $Z$ by its dense torus, there is a natural map
\[
\mathscr L_\Gamma(Z)\to \mathscr L_\Gamma(\mathscr A_Z),
\]
where the latter is the space of prestable logarithmic maps to the Artin fan. This space is a logarithmically smooth Artin stack~\cite{AW}. Moreover, the toroidal skeleton of this space is naturally identified with the moduli space of all (not necessarily realizable {or even balanced}) tropical maps from genus $1$ curves to the fan of $Z$, see~\cite{R16}. The locus of realizable curves is a sublocus in the moduli space. After subdividing this cone complex, this sublocus is supported on a subcomplex. This subdivision induces a birational modification of $\mathscr L_\Gamma(\mathscr A_Z)$, and thus a modification of $\mathscr L_\Gamma(Z)$. This modification can naturally be identified with the moduli of well-spaced logarithmic maps $\mathcal W_\Gamma(Z)$ defined above. The radial alignments developed in~\cite{RSW17A} and recalled in Section~\ref{sec: prelims-from-prequel} give rise to the modular interpretation.

The construction of $\mathcal W_\Gamma(Z)$ is not a formal lifting of our previous results on ordinary stable maps to the logarithmic category~\cite{RSW17A}. Given an absolute genus $1$ stable map $[C\to \PP^r]$, if no genus $1$ subcurve is contracted, then $[C\to \PP^r]$ is a smooth point of the moduli space. However, for a toric variety $Z$ and a genus $1$ logarithmic map $[C\to Z]$, the deformations of the map can be obstructed even if no component of $C$ is contracted. This is true even if $Z= \PP^r$ with its toric logarithmic structure. This behaviour is akin to the genus $1$ absolute stable maps theory for semipositive targets. While the tangent bundle of $\PP^r$ is ample, the logarithmic tangent bundle of a toric variety is trivial. This allows for a larger space of obstructions to deforming genus $1$ logarithmic maps than in the absolute theory. We overcome this by identifying and forcing the stronger factorization property above. 

\subsection{Tropical enumerative geometry and realizability} The realizability problem for tropical curves is a combinatorial shadow of the problem of characterizing the closure of the main component in the space of logarithmic maps. The difficulty of the problem has limited tropical enumerative techniques to low target dimensions~\cite{BBM14,CJM1,CMR14a, Mi03} or to genus $0$ curves~\cite{Gro14,Gro15,MR16,NS06,R15b}.

In the higher genus, higher dimensional situation, there are two directions in which one may generalize the picture above. The first is to develop a systematic method to decompose logarithmic Gromov--Witten invariants, as a sum of virtual invariants over tropical curves~\cite{ACGS15,Par11,R19}. The second is to analyze the tropical lifting problem and produce a ``reduced'' curve counting theory that captures the principal component contribution to the virtual count. This paper addresses the second of these in genus one. The realizability theorem in genus $1$ allows us to decompose these reduced invariants of any toric variety over tropical curves. The degeneration formula for these invariants is work that we hope to return to. Note that the analogous problem for smooth pairs has recently been treated in~\cite{BNR19}. 

There have been a number of interesting partial results on tropical realizability in the last decade, thanks to the efforts of many~\cite{BPR16,CFPU,JR17,KatLift,Mi03,Ni09,NS06,R16,R15a,Sp07}. The genus $1$ story alone has seen heavy interest. Speyer identified the sufficiency of a strong form of well-spacedness condition for superabundant genus $1$ tropical curves using Tate's uniformization theory. Using the group law on the analytification of an elliptic curve, Baker--Payne--Rabinoff show that a weaker condition was necessary. The existence of genus $1$ tropical curves which failed Speyer's condition but were nonetheless realizable was established in~\cite{R16}. 

In higher genus, very few results are known. That non-superabundant higher genus tropical curves are realizable was established by Cheung--Fantini--Park--Ulirsch~\cite{CFPU}, and limits of realizable curves can be shown to be realizable~\cite{R15a,R16}. Katz showed that the logarithmic tangent/obstruction complex for degenerate maps gives rise to necessary combinatorial conditions for realizability in higher genus, including a version of well-spacedness~\cite{KatLift}. These methods do not prove sufficiency in any cases. A sufficient condition for realizability for some superabundant chain of cycles geometries has recently been shown to hold and used to establish new results in Brill--Noether theory~\cite{JR17}.

\subsection{User's guide} We have written this paper so it may be read independently from the prequel, in which the space of ordinary stable maps to $\mathbf P^r$ was considered. In Section~\ref{sec: prelims-from-prequel} we recall the preliminary results on radial alignments and their contractions from~\cite{RSW17A}.  The moduli space of well-spaced logarithmic maps is constructed in Section~\ref{sec: well-spaced-logmaps} and the logarithmic unobstructedness appears as Theorem~\ref{thm:toric-log-smooth}. The tropical well-spacedness condition is discussed and defined precisely in Section~\ref{sec: well-spacedness}. Finally, tropical realizability results are restated in Theorem~\ref{thm: realizability} and proved in Section~\ref{sec: realizability-proof}.

\subsection*{Acknowledgements} Thanks are due to Dan Abramovich, Dori Bejleri, Sebastian Bozlee, Renzo Cavalieri, Dave Jensen, Eric Katz, Diane Maclagan, Davesh Maulik, David Speyer, Jeremy Usatine, and Ravi Vakil for helpful conversations and much encouragement. Special thanks are due to Sam Payne, who first explained the well-spacedness condition to D.R., and pointed out that a deformation theoretic understanding of the phenomenon was likely to be valuable. Luca Battistella and Navid Nabijou kindly pointed out errors in an earlier version of the paper, and we are grateful to them. The referee also provided valuable feedback. Finally, the authors learned that they were working on related ideas after a seminar talk of D.R. at Colorado State University when J.W. was in the audience; we thank Renzo Cavalieri for creating that opportunity. 

\subsection*{Funding} D.R. was partially supported by NSF grant DMS-1128155 (Institute for Advanced Study) and J.W. was partially supported by NSA Young Investigator's Grants H98230-14-1-0107 and H98230-16-1-0329.

\section{Preliminaries}

In this section, we recall some preliminaries on singularities of genus $1$ and on logarithmic and tropical geometry. There is some overlap between this section and the preliminary material appearing in the prequel to this article~\cite{RSW17A}, but we opt to include it for a more self-contained presentation.

\subsection{Genus 1 singularities}\label{sec: genus-1-singularities} 
Let $C$ be a reduced curve over an algebraically closed field $k$. For an isolated curve singularity $(C,p)$ with normalization $\pi: (\widetilde C,p_1,\ldots, p_m)\to (C,p)$, recall that $m$, the cardinality of $\pi^{-1} p$, is called the \textbf{number of branches of the singularity}. The \textbf{$\delta$-invariant} is defined as
\[
\delta: = \dim_k \bigl( \pi_\star (\mathscr O_{\widetilde C})/\mathscr O_C \bigr).
\]

Based on these two invariants, one defines the genus of $(C,p)$ as
\[
g = \delta-m+1.
\]

We will frequently make use of the \textbf{seminormalization} of $(C,p)$ in our arguments. The \textbf{seminormalization} is a partial resolution of $(C,p)$ to a singularity of genus $0$ that is homeomorphic to $(C,p)$. Explicitly, equip the underlying topological space of $(C,p)$ with the subring $\mathscr A$ of regular functions on the normalization $\widetilde C$ that are well-defined on the underlying topological space of $C$. In particular, there are $g$ additional conditions required for a function in $\mathscr A$ to descend to $(C,p)$, i.e.,
\begin{equation*}
g = \dim_k \bigl( \mathscr A / \mathscr O_C \bigr) .
\end{equation*}

Let $E$ be a proper Gorenstein curve of genus~$1$, smooth away from a unique genus $1$ singularity.  Let $\nu : F \to E$ be the seminormalization and let $\mu : G \to F$ be the normalization.  We have inclusions:
\begin{gather*}
\mathscr O_E \subset \nu_\star \mathscr O_F \subset \nu_\star \mu_\star \mathscr O_G \subset K \\
J \supset \omega_E \supset \nu_\star \omega_F \supset \nu_\star \mu_\star \omega_G
\end{gather*}
Here $K$ is the sheaf of meromorphic functions on $E$ and $J$ is the sheaf of meromorphic differentials.  For each $X = E, F, G$, the pairs $\omega_X$ and $\mathscr O_X$ are dual to other another with respect to the residue pairing $K \otimes J \to k$, in the sense that each is the annihilator of the other~\cite[Proposition~1.16~(ii)]{AK}.

Consider the exact sequence~\eqref{eqn:seminorm-seq}:
\begin{equation} \label{eqn:seminorm-seq}
0 \to \mathscr O_E \to \nu_\star \mathscr O_F \to \nu_\star(\mathscr O_F) / \mathscr O_E \to 0
\end{equation}
In the long exact cohomology sequence~\eqref{eqn:seminorm-les}
\begin{equation} \label{eqn:seminorm-les}
0 \to H^0(E, \mathscr O_E) \to H^0(F, \mathscr O_F) \to \nu_\star(\mathscr O_F) / \mathscr O_E \to H^1(E, \mathscr O_E) \to H^1(F, \mathscr O_F)
\end{equation}
the map $H^0(E, \mathscr O_E) \to H^0(F, \mathscr O_F)$ is an isomorphism because both $E$ and $F$ are proper, connected, and reduced; furthermore $H^1(F, \mathscr O_F) = 0$ since $F$ has genus~$0$.  By Serre duality, $H^1(E, \mathscr O_E)$ is dual to $H^0(E, \omega_E)$.  Since both are $1$-dimensional, the choice of a nonzero $\alpha \in H^0(E, \omega_E)$ induces an isomorphism $H^1(E, \mathscr O_E) \to k$.  The composition
\begin{equation*}
\nu_\star(\mathscr O_F) / \mathscr O_E \to H^1(E, \mathscr O_E) \to k
\end{equation*}
may be identified with the residue pairing, sending $f \bmod{\mathscr O_E}$ to $\res f\alpha$.  This follows, for example, by the construction of the dualizing sheaf in~\cite[Remark~1.9 and Remark~1.12]{AK}.

We know that $\omega_F / \mu_\star(\omega_G)$ is spanned by the differentials
\begin{equation} \label{eqn:omega-F}
\frac{dx_i}{x_i} - \frac{dx_j}{x_j}
\end{equation}
where the $x_i$ are local coordinates of the branches of $E$ at the singular point.  As $\omega_E / \nu_\star(\omega_F)$ is $1$-dimensional, $\omega_E$ is generated relative to $\nu_\star (\omega_F)$ by a differential of the following form:
\begin{equation} \label{eqn:omega-E}
\sum_i \frac{c_i d x_i}{x_i^2} + \frac{c' dx_1}{x_1} 
\end{equation}
If $f \in \mathscr O_E$ has the expansion $f(0) + b_i x_i + \cdots$ on the $i$th component of $F$ then this differential imposes the constraint
\begin{equation*}
c' f(0) + \sum b_i c_i = 0 .
\end{equation*}
In order for $E$ to be Gorenstein, $\omega_E$ must be a line bundle, so the generators~\eqref{eqn:omega-F}  of $\omega_F$ must be multiples of the generator~\eqref{eqn:omega-E}.  This immediately implies $c' = 0$ and that all of the $c_i$ are nonzero.  Conversely, if $c' = 0$ and all of the $c_i$ are nonzero, then $c_j x_i - c_i x_j \in \mathscr O_E$ and
\begin{equation*}
(c_j x_i - c_i x_j) \sum_k \frac{c_k d x_k}{x_k^2} = c_j c_i \frac{d x_i}{x_i} - c_i c_j \frac{d x_j}{x_j}
\end{equation*}
implies that the generators~\eqref{eqn:omega-F} are multiples of~\eqref{eqn:omega-E}.  This proves the following proposition:

\begin{proposition} \label{prop:dualizing-generator}
If $E$ is a Gorenstein curve with a genus~$1$ singularity then $\omega_E$ is generated in a neighborhood of its singular point by a meromorphic form~\eqref{eqn:omega-E}, with $c' = 0$, where the $x_i$ are local parameters for the branches of $E$ at the singular point.
\end{proposition}

By consideration of the residue condition imposed by the form~\eqref{eqn:omega-E}, we can also obtain a local description of the Gorenstein, genus~$1$ curve singularities.  A more conceptual proof of this result can be found in \cite[Proposition~A.3]{Smyth}.

\begin{proposition}
For each integer $m\geq 0$, there exists a unique Gorenstein singularity $(C,p)$ of genus $1$ with $m$ branches. If $m = 1$ then $(C,p)$ can be identified with the cusp $\mathbf{V}(y^2-x^3)$, if $m = 2$ then $(C,p)$ can be identified with the ordinary tacnode $\mathbf V(y^2-yx^2)$, and if $m\geq 3$, then $(C,p)$ is the germ at the origin of the union of $m$ general lines through the origin in $\mathbf A^{m-1}$. 
\end{proposition}



\subsection{Tropical curves} We follow standard conventions and definitions for tropical curves and tropical stable maps.

\begin{definition}
An \textbf{$n$-marked tropical curve} $\plC$ is a finite graph $G$ with vertex and edge sets $V$ and $E$, enhanced by
\begin{enumerate}
\item a \textbf{marking function} $m: \{1,\ldots,n\}\to V$,
\item a \textbf{genus function} $g:V\to \NN$,
\item a \textbf{length function} $\ell: E\to \RR_{+}$.
\end{enumerate}
The \textbf{genus} of a tropical curve $\plC$ is defined to be 
\[
g(\plC) = h_1(G)+\sum_{v\in V} g(v)
\]
where $h_1(G)$ is the first Betti number of the geometric realization of $\plC$. An $n$-marked tropical curve is \textbf{stable} if (1) every genus $0$ vertex has valence at least $3$ and (2) every genus $1$ vertex has valence at least $1$. 
\end{definition}


More generally, one may permit the length function $\ell$ above to take values in an arbitrary toric monoid $P$. This presents us with a natural notion of a family of tropical curves.

\begin{definition}
Let $\sigma$ be a rational polyhedral cone with dual cone $S_\sigma$. A \textbf{family of $n$-marked prestable tropical curves over $\sigma$} is a tropical curve whose length function takes values in $S_\sigma$.
\end{definition}

We note that given a tropical curve over $\sigma$, each point of $\sigma$ determines a tropical curve in the usual sense. Indeed, choosing a point of $\sigma$ is equivalent to choosing a monoid homomorphism
\[
\varphi: S_\sigma \to \RR_{\geq 0}.
\]
Applying this homomorphism to the edge length $\ell(e)\in S_\sigma$ produces a real and positive length for each edge.

\subsection{Logarithmic geometry: working definitions} Let $N$ be a free abelian group of finite rank and $X^\circ$ be a subscheme of a torus $T = \mathbf G_m\otimes N$ over a field $k$ equipped with the trivial valuation. Let $K$ be a valued field extending $K$, with valuation surjective onto $\RR$. Then, the tropicalization of $X$ is the image of $X(K)$ under the coordinatewise valuation map
\[
T(K)\to \RR\otimes N.
\]
This set is denoted $X^{\trop}$, and can be given the structure of a fan. This fan distinguishes a partial compactification of $T$ to a toric variety $Y$. The embedding of the closure $ X\hookrightarrow Y$ determines, locally on $X$, a natural class of \textit{monomial} functions obtained by restricting the monomials on $T$. These monomials form a sheaf of monoids $M_X$ under multiplication, and a tautological map of monoids
\[
\mathcal O_X^\star\subset M_X\to \mathcal O_X.
\]
The quotient is another sheaf of monoids $\overline M_X:=M_X/\mathcal O_X^\star$, and amounts to considering monomial functions up to scalars. 

Sections of the groupification $\overline M_X^{\mathrm{gp}}$ can be interpreted as piecewise linear functions on $X^{\trop}$. Just as in the toric case, piecewise linear functions on $X^{\trop}$ give rise to line bundles on $X$. Specifically, given a piecewise linear function, the set of algebraic lifts of it in $\overline M^{\mathrm gp}_X$ form a torsor under the multiplicative group, and therefore a line bundle. This is explained more precisely in Section~\ref{sec: lb-from-pl} below.

A logarithmic scheme is an object that possesses the main features present above. The requirement that $X$ be embedded in a toric variety can be dropped. Instead, one need only assume that $X$ (locally) admits a morphism to a toric stack.  The data of the sheaf $\overline M_X$ may be thought of as the sheaf of piecewise linear functions on $X$.

To be more precise, it is convenient to reverse the logical order and specify the monomials first. Given a scheme $S$, a logarithmic structure is a sheaf of monoids $M_S$ in its \'etale topology and sharp homomorphism $\varepsilon : M_S \to \mathcal O_S$ (the codomain given its multiplicative monoid structure).  Sharpness means that each local section of $\mathcal O_S^\star$ has a unique preimage along $\varepsilon$.  The quotient $M_X / \varepsilon^{-1} \mathcal O_X^\star$ is called the \emph{characteristic monoid} and is denoted $\overnorm M_X$ with its operation denoted \emph{additively}; the image of section $\alpha$ of $M_X^{\rm gp}$ in $\overnorm M_X^{\rm gp}$ is denoted $\overnorm\alpha$.  We assume all logarithmic structures are integral ($M_X$ is contained in its associated group $M_X^{\rm gp}$) and saturated (if $\alpha \in M_X^{\rm gp}$ and $n\alpha \in M_X$ for some integer $n \geq 1$ then $\alpha \in M_X$).

Such objects may be assembled into a category. The category of logarithmic schemes has the analogous constructions and notions from scheme theory, keeping track of the tropical data through the sheaves of piecewise linear functions.

For more of the general theory of logarithmic structures, we refer the reader to Kato's original article~\cite{Kat89}. A detailed study of the relationship between tropical and logarithmic geometry from a categorical point of view is undertaken in~\cite{CCUW}.

\subsection{Curves {\it \&} logarithmic structures} Let $(S,M_S)$ be a logarithmic scheme. A \textbf{family of logarithmically smooth curves over $S$} is a logarithmically smooth, flat, and proper morphism
\[
\pi: (C,M_C) \to (S,M_S),
\]
with connected and reduced geometric fibers of dimension $1$. We recall F. Kato's structure theorem for logarithmic curves~\cite{Kat00}.

\begin{theorem}
Let $C\to S$ be a family of logarithmically smooth curves. If $x\in C$ is a geometric point, then there is an \'etale neighborhood of $C$ over $S$, with a strict morphism to an \'etale-local model $\pi:V\to S$, and $V\to S$ is one of the following:
\begin{itemize}
\item { (the smooth germ)} $V  = \mathbf A^1_S\to S$, and the logarithmic structure on $V$ is pulled back from the base;
\item { (the germ of a marked point)} $V = \mathbf A^1_S\to S$, with logarithmic structure pulled back from the toric logarithmic structure on $\mathbf A^1$;
\item { (the node)} $V = \mathscr O_S[x,y]/(xy = t)$, for $t\in \mathscr O_S$.  The logarithmic structure on $V$ is pulled back from the multiplication map $\mathbf A^2 \to \mathbf A^1$ of toric varieties along a morphism $t : S \to \mathbf A^1$ of logarithmic schemes.
\end{itemize}

The image of $t\in M_S$ in $\overnorm M_S$ is referred to as the \textbf{deformation parameter of the node.}


\end{theorem}

Associated to a logarithmic curve $C\to S$ is a family of tropical curves.

\begin{definition}
Let $C\to S$ be a family of logarithmically smooth curves and assume that the underlying scheme of $S$ is the spectrum of an algebraically closed field. Then, \textbf{the tropicalization $C$}, denoted $\plC$, is obtained as follows: (1) the underlying graph is the marked dual graph of $C$ equipped with the standard genus and marking functions, and (2) given an edge $e$, the generalized length $\ell(e) = \delta_e\in \overnorm M_S$ is the deformation parameter of the corresponding node of $C$.
\end{definition}


For more about logarithmic curves and their relationship to tropical curves, the reader may consult~\cite{CCUW}. 

\subsection{Geometric interpretation of the sections of a logarithmic structure}
\label{sec:char-sect} 
Given a logarithmic curve $C\to S$, it will be helpful to interpret sections of the sheaves $M_C^{\rm gp}$, and $\overnorm M_C^{\rm gp}$ geometrically. 

\numberwithin{theorem}{subsubsection}
\subsubsection{The affine and projective lines} 
\label{sec:monoid-sections}

Let $(X,\varepsilon: M_X\to \mathscr O_X)$ be a logarithmic scheme.  A section of $M_X$ corresponds to a map $X \to \mathbf A^1$, the target given its toric logarithmic structure. Let $\alpha$ be such a section and $\overnorm\alpha$ be its image in $\overnorm M_X$.  Then $\varepsilon(\alpha)$ is a unit if and only if $\overnorm\alpha = 0$.

With its logarithmic structure, $\PP^1$ can be constructed as the quotient of $\mathbf A^2 - \{ 0 \}$ by $\Gm$.  Any map $X \to \PP^1$ lifts locally to $\mathbf A^2 - \{ 0 \}$ and can therefore be represented by a pair of sections $(\xi,\eta)$ of $M_X$.  The ratio $\xi^{-1} \eta$, which is a section of $M_X^{\rm gp}$, is invariant under the action of $\Gm$, since $\Gm$ acts with the same weight on $\xi$ and $\eta$.

Therefore a map $X \to \PP^1$ gives a well-defined section $\alpha$ of $M_X^{\rm gp}$.  Not every section of $M_X^{\rm gp}$ arises this way, because the map $(\xi, \eta) : X \to \mathbf A^2$ from which $\alpha$ was derived could not meet the origin.  This condition implies that, for each geometric point $x$ of $X$, either $\overnorm\xi_x = 0$ or $\overnorm\eta_x = 0$.  In terms of $\overnorm\alpha$, this means that $\overnorm\alpha_x \geq 0$ or $\overnorm\alpha_x \leq 0$.  We term this property being \textbf{locally comparable to $0$}.

{Our observations prove the following proposition:}

\begin{proposition}
Let $X$ be a logarithmic scheme.  Maps $X \to \PP^1$, the latter given its toric logarithmic structure, may be identified with sections $\alpha$ of $M_X^{\rm gp}$, whose images $\overnorm\alpha$ in $\overnorm M_X^{\rm gp}$ are locally comparable to $0$.
\end{proposition}

Because it has charts, the sheaf $\overnorm M_X^{\rm gp}$ locally admits a surjection from a constant sheaf, so the condition on $\overnorm\alpha$ in the proposition is open on the base:  if $X$ is a family of logarithmic schemes over $S$ and a section $\alpha$ of $\overnorm M_X^{\rm gp}$ verifies $\overnorm\alpha \geq 0$ or $\overnorm\alpha \leq 0$ for all $x$ in a geometric fiber $X_s$ of $X$ over $S$ then it also verifies that condition for all $t$ in some open neighborhood of $s$.

This observation is particularly useful for studying infinitesimal deformations of logarithmic maps to $\PP^1$, as it is equivalent to deform the section $\alpha$ of $M_X^{\rm gp}$.

\begin{definition}
For any logarithmic scheme $X$, we define $\mathbf{G}_{\log}(X) = \Gamma(X, M_X^{\rm gp})$.  Identifying $X$ with its functor of points, we also write $\mathbf{G}_{\log}(X) = \Hom(X, \mathbf{G}_{\log})$.
\end{definition}

\begin{remark}
The functor $\mathbf{G}_{\log}$ is not representable by a logarithmic scheme; it is analogous to an algebraic space (see \cite{logpic}, for example).  The above considerations may be seen as a demonstration that $\PP^1$ is logarithmically \'etale over $\mathbf{G}_{\log}$. {A discussion of the simpler spaces rational curves in $\mathbf{G}_{\log}$ may be found in~\cite{RW19}.}

We prefer to avoid a discussion of the geometric structure of $\mathbf{G}_{\log}$ in this paper.  The reader should feel free to regard maps to $\mathbf{G}_{\log}$ as a convenient shorthand for sections of $M_X^{\rm gp}$ and nothing more.  {The advantage of treating $\mathbf G_{\log}$ as a geometric object, and not merely an abstract sheaf, and working with that object instead of geometric models like $\mathbf P^1$, is that the latter approach necessitates an apparently endless process of subdivision and refinement that obscures the geometric essence of our arguments.} 
\end{remark}

\subsubsection{Maps to toric varieties}
\label{sec:to-toric}

The observations above concerining logarithmic maps to $\PP^1$ may be extended to all toric varieties.  Indeed, if $Z = \Spec k[S_\sigma]$ is an affine toric variety defined by a cone $\sigma$ and character lattice $N^\vee$, then there is a canonical map
\begin{equation*}
S_\sigma \to \Gamma(Z, M_Z) ,
\end{equation*}
which extends to a map
\begin{equation*}
N^\vee \to \Gamma(Z, M^{\rm gp}_Z) .
\end{equation*}
The construction of this map commutes with restriction to open torus invariant subvarieties, and therefore glues to a well-defined map on any toric variety.

\begin{proposition}
Let $X$ be a logarithmic scheme and let $Z$ be a toric variety with fan $\Sigma$ and character lattice $N^\vee$.  Morphisms $X \to Z$ may be identified with morphisms $N^\vee \to \Gamma(X, M_X^{\rm gp})$ such that, for each geometric point $x$ of $X$, there is a cone $\sigma \in \Sigma$, such that the map
\begin{equation*}
S_\sigma \to \Gamma(X, M_X^{\rm gp}) \to \Gamma(X, \overnorm M_X^{\rm gp}) \to \overnorm M_{X,x}^{\rm gp}
\end{equation*}
factors through $\overnorm M_{X,x}$.
\end{proposition}

\begin{definition}
Let $N$ be a finitely generated free abelian group.  We write $(N \otimes \mathbf{G}_{\log})(X) = \Hom(N^\vee, \Gamma(X, M_X^{\rm gp}))$ and use $\Hom(X, N \otimes \mathbf{G}_{\log})$ for the same notion.
\end{definition}

\begin{remark}
The discussion above shows that, if $Z$ is a toric variety with cocharacter lattice $N$, then there is a canonical logarithmic modification $Z \to N \otimes \mathbf{G}_{\log}$.
\end{remark}

\subsubsection{Sections of the characteristic monoid}
\label{sec:char-mon}

Since logarithmic maps $X \to \mathbf A^1$ correspond to sections of $M_X$, maps $X \to [ \mathbf A^1 / \Gm ]$ correspond to sections of $M_X / \mathscr O_X^\star = \overnorm M_X$.  The quotient $[ \mathbf A^1 / \Gm ]$ is usually denoted $\mathscr A$ and is called the \textbf{Artin fan} of $\mathbf A^1$.

It is shown in \cite[Remark~7.3]{CCUW} that, if $X$ is a logarithmic curve over $S$, and the underlying scheme of $S$ is the spectrum of an algebraically closed field, then sections of $\overnorm M_X$ (which is to say, maps $X \to \mathscr A$) may be interpreted as piecewise linear functions on the tropicalization of $X$ that are valued in $\overnorm M_S$ and are linear along the edges with integer slopes.

Similar reasoning, combined with the discussion in Section~\ref{sec:monoid-sections} shows that maps $X \to [ \mathbf P^1 / \Gm ]$ correspond to sections $\alpha$ of $\overnorm M_X^{\rm gp}$ that are locally comparable to $0$.  If $X$ is a curve, then these sections are the piecewise linear functions on the tropicalization that are everywhere valued in $\overnorm M_S$ or in $- \overnorm M_S$.


\begin{remark}
Even though its underlying ``space'' is an algebraic stack, $[\mathbf A^1 / \Gm]$ represents a \textbf{functor} on logarithmic schemes. { This contrasts with the more common situation, where algebraic stacks typically only represent categories fibered in groupoids over schemes.}
\end{remark}

\numberwithin{theorem}{subsection}
\subsection{Line bundles from piecewise linear functions}\label{sec: lb-from-pl} For any logarithmic scheme $X$, there is a short exact sequence
\begin{equation*}
0 \to \mathscr O_X^\star \to M_X^{\rm gp} \to \overnorm M_X^{\rm gp} \to 0
\end{equation*}
of the sheaves associated to the logarithmic structure. Given a section $\alpha \in \Gamma(X, \overnorm M_X^{\rm gp})$, the image of $\alpha$ under the coboundary map
\begin{equation*}
H^0(X, \overnorm M_X^{\rm gp}) \to H^1(X, \mathscr O_X^\star)
\end{equation*}
is represented by an $\mathscr O_X^\star$-torsor $\mathscr O_X^\star(-\alpha)$ on $X$ and gives rise to an associated line bundle.  Thus, to each piecewise linear function $f$ on $\plC$ that is linear on the edges with integer slopes and takes values in $\overnorm M_S$, we have an associated line bundle $\mathscr O(-f)$.

The explicit line bundle obtained by this construction is recorded in~\cite[Section 2]{RSW17A}.

\numberwithin{theorem}{subsection}
\subsection{Tropicalization of morphisms to toric varieties}
\label{sec:tropicalization}

Let $Z$ be a toric variety with dense torus $T$, equipped with its standard logarithmic structure, and let $N$ and $N^\vee$ be the cocharacter and character lattices of $Z$.


Let $C$ be a logarithmic curve over $S$, and assume that the underlying scheme of $S$ is the spectrum of an algebraically closed field.  A logarithmic map $\varphi : C \to Z$ induces a map 
\begin{equation} \label{eqn:map-to-tropGm}
N^\vee \to \Gamma(Z, \overnorm M_Z) \to \Gamma(C, \overnorm M_C) 
\end{equation}
by the discussion in Section~\ref{sec:to-toric}.

As remarked in Section~\ref{sec:char-sect}, the sections of $\overnorm M_C$ are piecewise linear functions on the tropicalization $\plC$ of $C$ that are linear with integer slopes along the edges and are valued in $\overnorm M_S^{\rm gp}$.  If we assume in addition that $\overnorm M_S = \RR_{\geq 0}$ then we obtain a piecewise linear map
\begin{equation*}
\plC \to \Hom(N^\vee,\RR) = N_\RR
\end{equation*}
that we call the \emph{tropicalization} of $C \to Z$. It will sometimes be convenient to think of this as a map from $\plC \to \Sigma$, where $\Sigma$ is the fan of $Z$.

\begin{lemma}
The map $\plC \to N_\RR$, constructed above, satisfies the balancing condition.
\end{lemma}
\begin{proof}
This is proved in~\cite[Section 1.4]{GS13}.
%
%
\end{proof}

\subsection{Minimality}
\label{sec:minimality}

\numberwithin{theorem}{subsubsection}
\subsubsection{Minimal logarithmic structures}
\label{sec:minlog}

 A crucial concept in the theory of logarithmic moduli problems is that of \textbf{minimality}. Let $\mathbf{LogSch}$ denote the category of fine and saturated logarithmic schemes. Given a moduli stack $\fM$ over $\mathbf{LogSch}$ and a logarithmic scheme $S$, the fiber $\fM(S)$ of the fibered category $\fM$ over $S$ is the groupoid logarithmic geometric objects $[\mathscr X\to S]$ defined over $S$, as specified by the moduli problem. 

Logarithmic geometric objects are algebraic schemes or stacks with the additional structure of a sheaf of monoids. The description of $\fM$ as a category fibered in groupoids over $\mathbf{LogSch}$ does not furnish such an object: if $\underline S$ is a scheme without a chosen logarithmic structure, it does not make mathematical sense to consider the fiber of $\fM$ over $\underline S$. Said differently, there is no ``underlying scheme, or underlying stack, or underlying category fibered in groupoids over schemes'' of $\fM$.

The difficulty that must be overcome is that given an ordinary scheme $\underline S$, there are many choices for logarithmic schemes $(S,M_S)$ enhancing $\underline S$, and it is unclear which one to pick.  The notion of minimality, introduced by F. Kato and recently clarified and expanded~\cite{AC11,Che10,Gi12,GS13,Wis16a,Wis16b} identifies the correct logarithmic structures to allow on $\underline S$ as those satisfying a universal property, recalled below.

Assuming that $\fM$ does have an underlying scheme, we arrive at a \emph{necessary} condition for $\fM$ to be representable by a logarithmic scheme.  Suppose that $S \to \fM$ is a morphism of logarithmic schemes then the logarithmic structure of $\fM$ pulls back to a logarithmic structure $M$ on the underlying scheme $\underline S$ of $S$.  Moreover there is a factorization
\begin{equation*} \xymatrix@R=10pt{
	S \ar[dr] \ar[dd] \\
	& \fM \\
	(\underline S, M) \ar[ur]
} \end{equation*}
that is \emph{final} among all such factorizations.  This finality condition can be phrased entirely in terms of the moduli problem defining $\fM$, and Gillam shows that if minimal factorizations exist for all $S \to \fM$, and are preserved by base change, then $\fM$ comes from a logarithmic structure on a moduli problem over \emph{schemes}~\cite{Gi12,Wis16a}.  

\begin{theorem}[Gillam] \label{thm:gillam}
When $\mathfrak M$ is a category fibered in groupoids over logarithmic schemes that comes from a logarithmic structure on a category fibered in groupoids $\mathfrak N$ over schemes, $\mathfrak N$ can be recovered from $\mathfrak M$ as the subcategory of minimal objects.
\end{theorem}

Throughout this paper, we present logarithmic moduli problems and indicate monoidal and tropical (see Section~\ref{sec:mintrop}) characterizations of their minimal objects to recover the underlying schematic moduli problems.

\subsubsection{Minimality as tropical representability}
\label{sec:mintrop}

We explain the concept in the case of stable maps for concreteness, where it becomes a tropical concept. This expands on~\cite[Remark~1.21]{GS13}. Let $\fM_{g,n}(Z)$ denote stack over $\mathbf{LogSch}$ parametrizing logarithmic maps from genus $g$, $n$-pointed curves to a toroidal scheme $Z$. Let $\Sigma$ be the fan of $Z$.

Let $S$ be a standard logarithmic point $\spec(\NN\to \CC)$ and let $[C\to Z]$ be a logarithmic map over $S$. As explained in Section~\ref{sec:tropicalization}, the morphisms on sheaves of monoids may be dualized to produce a tropical map
\[
\plC\to \Sigma.
\]
Replacing $\NN$ with an arbitrary toric monoid, one obtains a \textbf{family} of tropical maps. 

From our discussion of minimality, we see that given a logarithmic stable map over $\spec(P\to \CC)$, the monoid be cannot be arbitrary, since by pulling back via a morphism $P\to \RR_{\geq 0}$, we must obtain a tropical map. With this observation, there is a clear choice for a universal $P^{\min}$ such that all other enhancements $\spec(P\to \CC)$ of the same underlying map must be pulled back from $\spec(P^{\min}\to \CC)$. That is, we may choose $P^{\min}$ to be the monoid whose dual cone $\Hom(P^{\min},\RR_{\geq 0})$ is the cone of \textbf{all} tropical maps of the given combinatorial type.  {Succinctly, a logarithmic structure is minimal for a given moduli problem if it represents the \emph{tropical deformation space}.}

In Figure~\ref{fig: non-minimal-family} below, taking $Z = pt$, we depict the duals of the characteristic monoid on the base of a non-minimal family. If one drops the condition that $\ell_1$ and $\ell_2$ coincide, we obtain the corresponding minimal monoid.

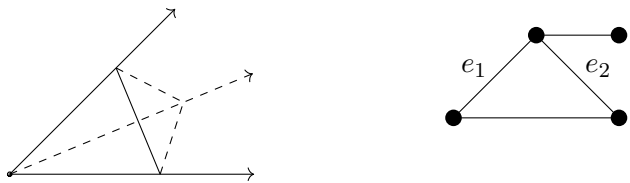
\begin{figure}[ht!]
\begin{tikzpicture}

\draw [ball color=black] (0,0) circle (0.3mm);
\draw [->] (0,0)--(3.25,0);
\draw [->] (0,0)--(2.2,2.2);
\draw [dashed,->] (0,0)--(22.5:3.5); 

\draw ({sqrt(2)},{sqrt(2)})--(2,0);
\draw [dashed] ({sqrt(2)},{sqrt(2)})--(22.5:2.5);
\draw [dashed] (2,0)--(22.5:2.5);

\begin{scope}[xshift=7cm, yshift=0.75cm, scale=1.1]
\coordinate (2) at (0,1);
\coordinate (3) at (-1,0);
\coordinate (4) at (1,0);
\coordinate (1) at (1,1);

\draw (1)--(2)--(3)--(4)--(2);

\fill (2) circle (1mm);
\fill (3) circle (1mm);
\fill (4) circle (1mm);
\fill (1) circle (1mm);

\node at (-0.75,0.6) {$e_1$};
\node at (0.75,0.6) {$e_2$};
\end{scope}

\end{tikzpicture}
\caption{Consider the cone of tropical curves whose underlying graph is shown on the right, such that the edge lengths of $e_1$ and $e_2$ are equal. This cone is $3$-dimensional. An associated family of logarithmic curves whose minimal monoid is dual to this cone associated family of logarithmic curves is non-minimal, due to the relation that these two edge lengths coincide coincide.}\label{fig: non-minimal-family}
\end{figure}

Applying this reasoning at each geometric fiber gives a criterion to check whether any given family of logarithmic maps $C \to Z$ over a logarithmic scheme $S$ is minimal. With the minimal objects identified, we construct a moduli stack as a fibered category over \textbf{Sch}, whose fiber over a scheme $\underline S$ is the groupoid of \textbf{minimal} logarithmic maps over $S$. 

\numberwithin{theorem}{subsection}
\subsection{Preliminaries from the prequel: radial alignments}\label{sec: prelims-from-prequel} The results of this paper rely on the notion of a radially aligned logarithmic curve and its canonical contraction to a curve with elliptic singularities. These concepts were developed in the companion article~\cite{RSW17A}, and we briefly recall the statements that we require.

Let $S$ be a logarithmic scheme enhancing the spectrum of an algebraically closed field and let $C\to S$ be a logarithmic curve over it whose fibers have genus~$1$ and let $\plC$ be its tropicalization. Given an edge $e$, we write $\ell(e)\in \overnorm M_S$ for the generalized edge length of this edge. For each vertex $v\in \plC$, there is a unique path from the circuit of $\plC$, namely the smallest subgraph of genus $1$, to the chosen vertex $v$. Write this path as $e_1,\ldots, e_n$. Define 
\[
\lambda(v) = \sum_{i=1}^n \ell(e_i).
\]
The resulting function $\lambda$ is a piecewise linear function on $\plC$ with integer slopes, and thus, a global section of $\overnorm M_C$. When $S$ is a general logarithmic scheme and $\pi: C\to S$ a curve, this section glues along specialization morphisms to give rise to a well-defined and canonical global section in $\Gamma(S,\pi_\star \overnorm M_S)$. 

Given a logarithmic curve $C\to S$ and a geometric point $s\in S$, we let $\plC_s$ denote the corresponding tropical curve associated to $C_s$. Recall also that we view a monoid $P$ as being the positive elements in a partially ordered group, with the partial order defined by $a\geq b$ if $a-b\in P\subset P^{\mathrm{gp}}$.

\begin{definition}
We say that a logarithmic curve $C\to S$ is \textbf{radially aligned} if $\lambda(v)$ and $\lambda(w)$ are comparable for all geometric points $s$ of $S$ and all vertices $v, w \in \plC_s$.

We write ${\mathfrak M}_{1,n}^{\mathrm{rad}}$ for the category fibered in groupoids over logarithmic schemes whose fiber over $S$ is the groupoid of radially aligned logarithmic curves over $S$ having arithmetic genus~$1$ and~$n$ marked points.
\end{definition}

The following result is proved in~\cite[Section 3]{RSW17A}.

\begin{theorem}
The category of radially aligned, prestable, logarithmic curves of genus~$1$ with $n$~marked points is represented by an algebraic stack with logarithmic structure $\fM_{1,n}^{\mathrm{rad}}$. The natural map
\[
\fM_{1,n}^{\mathrm{rad}}\to \fM_{1,n},
\]
is a logarithmic blowup.
\end{theorem}

The second major construction in op.\ cit.\ is the construction of a contraction to a curve with elliptic singularities, from the data of a radially aligned curve with a chosen ``radius of contraction''. Let $C\to S$ be a radially aligned logarithmic curve of genus~$1$. We say that a section $\delta\in \overnorm M_S$ is \textbf{comparable to the radii of $C$} if for each geometric point $s\in S$, the section $\delta$ is comparable to $\lambda(v)$ for all vertices $v\in \plC_s$, in the monoid $\overnorm M_S$.

\begin{theorem}
Let $C\to S$ be a radially aligned logarithmic curve and $\delta\in \overnorm M_S$ a section comparable to the radii of $C$. Then, there exists a partial destablilization 
\[
\widetilde C\to C,
\]
and a contraction
\[
\widetilde C\to \overnorm C,
\]
where $\overnorm C\to S$ is a family genus~$1$ curves at worst Gorenstein genus $1$ singularities, {such that, for every geometric point $s$ of $S$, if $E$ is a component of $C_s$ such that $\lambda(E) < \delta_s$ then $E$ is contracted to a point in $\overnorm C$}.
\end{theorem}

An intuitive discussion of these concepts are presented in~\cite[Section~3.1]{RSW17A}. For working knowledge, reader may visualize the section $\delta$ as giving rise to a \textbf{circle of radius $\delta$} around the circuit of the tropical curve $\plC$. By subdividing the edges of $\plC$, one may produce a new tropical curve $\widetilde \plC$ such that every point of $\widetilde \plC$ at radius $\delta$ from the circuit is a vertex. This introduces valency $2$ vertices into the tropicalization, and induces the partial destabilization. By contracting the interior of the circle of radius $\delta$ in a versal family, one produces a curve with a Gorenstein singularity. 

\numberwithin{theorem}{section}
\section{Logarithmic maps to toric varieties}
\label{sec:log-maps-to-toric}

We construct the space of radially aligned logarithmic maps to a toric variety. The framework of radial alignments, together with the well-spacedness condition from tropical geometry, will lead to a proof of Theorem~\ref{thm: toric-targets}, which is the main result of this section. The symbol $Z$ will denote a proper toric variety with fan $\Sigma$. 

Recall that a morphism of polyhedral complex $\mathscr P\to \mathscr Q$ is a continuous map of the underlying topological spaces sending every polyhedron of $\mathscr P$ linearly to a polyhedron of $\mathscr Q$.

\begin{definition}\label{def: tsm}
A \textbf{tropical prestable map} or \textbf{tropical map} for short, is a morphism of polyhedral complexes
\[
F:\plC\to \Sigma
\]
where $\plC$ is an $n$-marked tropical curve, and the following conditions are satisfied.
\begin{enumerate}
\item For each edge $e\in \plC$, the direction of $F(e)$ is an integral vector. When restricted to $e$, the map has integral slope $w_e$, taken with respect to this integral direction. This integral slope is referred to as the \textbf{expansion factor} of $F$ along $e$. The expansion factor and primitive edge direction are together referred to as the \textbf{contact order} of the edge. 
\item The map $f$ is \textbf{balanced}: at all points of $\plC$ the sum of the directional derivatives of $F$ in each tangent direction is zero.
\end{enumerate}
The map is \textbf{stable} if it satisfies the following condition: if $p\in \plC$ has valence $2$, then the image of $\mathrm{Star}(v)$ is not contained in the relative interior of a single cone of $\Sigma$.
\end{definition}

Following Section~\ref{sec:tropicalization}, given a logarithmic prestable map to a toric variety 
\[
\begin{tikzcd}
(C,M_C) \arrow{d}\arrow{r}{f} & Z \\
(S,M_S), & 
\end{tikzcd}
\]
there is an associated family $\plC$ of tropical curves together with a map $[F: \plC\to \Sigma]$, satisfying the axioms of a tropical prestable map. 

\setcounter{subsection}{\value{theorem}}
\numberwithin{theorem}{subsection}
\subsection{Radial logarithmic maps} We begin with a construction of the stack of radially aligned logarithmic maps.

\begin{proposition} \label{prop:W-alg}
Let $Z$ be a toric variety.  There is an algebraic stack with logarithmic structure, $\mathfrak W(Z)$, parameterizing families of radially aligned curves $C$ and logarithmic morphisms $C \to Z$.  

The underlying algebraic stack of $\mathfrak W(Z)$ is locally quasifinite over the stack of ordinary prestable maps from radially aligned curves to $Z$, and its restriction to the open and closed substack of maps with fixed contact orders is quasifinite.
\end{proposition}

\begin{proof}
Let $\fM^{\mathrm{rad}}_{1,n}$ be the stack of radially aligned, $n$-marked, genus~$1$ logarithmic curves (Section~\ref{sec: prelims-from-prequel}) and let $C$ be its universal curve.  Then $\mathfrak W(Z)$ is the space of logarithmic prestable maps from $C$ to $Z$, and this is representable by an algebraic stack with a logarithmic structure~\cite[Corollary~1.1.1]{Wis16a}.  The local quasifiniteness is a consequence of~\cite[Theorem~1.1]{Wis16b} or~\cite[Proposition~3.6.3]{R15b}; under the assumption of fixed contact orders, the combinatorial types of a map $[C \to Z]$ are bounded~\cite[Theorem~3.8]{GS13}, and therefore $\mathfrak W(Z)$ is quasifinite over the space of maps of underlying schemes.
\end{proof}

Stability in $\mathfrak W(Z)$ is defined in terms of the underlying schematic map:

\begin{definition}
A radial map $[f:C\to Z]$ in $\mathfrak W(Z)$ over $\spec(\CC)$ is said to be \textbf{stable} if it satisfies the following conditions. 
\begin{enumerate}
\item If $D\subset C$ is an irreducible component of genus $0$ contracted by $f$, then $D$ supports at least $3$ special points.
\item If $C$ is a smooth curve of genus $1$, then $C$ is not contracted. 
\end{enumerate}
A family of ordered logarithmic maps is stable if each geometric fiber is stable. 
\end{definition}


\setcounter{subsubsection}{\value{theorem}}
\subsubsection{Minimal monoids}

We give a tropical description of the logarithmic structure of $\mathfrak W(Z)$.  We leave it to the reader to verify that this description is correct, either using \cite[Appendix~C.3]{Wis16a} or adapting the arguments from~\cite[Section 3]{Che10} or~\cite[Proposition~1.22]{GS13}.


The minimality condition may be checked on geometric fibers, so we assume that the underlying scheme of $S$ is $\Spec C$.  Let $\sigma_S$ be the corresponding dual cone $\Hom(\overnorm M_S,\RR_{\geq 0})$ of the characteristic monoid of $S$.  By forgetting the alignment, a radial map $[f]$ above produces a usual logarithmic map with combinatorial type $\Theta$. Letting $\sigma_\Theta$ be the associated cone of tropical maps, we have a morphism of cones
\[
\sigma_S \to \sigma_\Theta.
\]
In the tropical moduli cone $\sigma_\Theta$ above, the locus of tropical curves whose vertices are ordered in the same manner as $C$ forms a cone $\sigma(f)$. 

\setcounter{theorem}{\value{subsubsection}}
\begin{definition}
Let $f:C\to Z$ be a family of ordered logarithmic maps over a logarithmic base $S$. The map $[f]$ is a \textbf{minimal} ordered logarithmic map if for each geometric point $\overnorm s\in S$, there is an isomorphism of cones
\[
\Hom(\overnorm M_{S,\overnorm s},\RR_{\geq 0})\cong \sigma(f_{\overnorm s}). 
\]
\end{definition}


\subsection{The factorization property}
\label{sec:factorization}
To detect the curves that smooth to the main component, we will need to identify certain contractions of the source curve constructed from the tropical maps and use the methods developed in~\cite{RSW17A}.

Let $\overnorm C$ be a Gorenstein curve of arithmetic genus $1$. We will refer to $E$, the smallest connected subcurve of $C$ of arithmetic genus $1$, as the \textbf{circuit component} of $C$. Given a family $C\to S$, we give the nodes and markings the standard logarithmic structure, and we give $C$ the trivial logarithmic structure near any genus~$1$ singularities.

Given an aligned logarithmic curve $C$ of genus $1$ and a contraction $C\to \overnorm C$, we may equip $\overnorm C$ with the logarithmic structure defined above. This enhances $C\to \overnorm C$ to a logarithmic morphism.

Let $(C,M_C)\to (S,M_S)$ be a radially aligned logarithmic curve and let $Z$ be a toric variety with cocharacter lattice $N$.  We associate a section $\delta_f \in \overnorm M_S$ to a logarithmic map $f:C\to Z$ over $S$. Let $\plC$ be the tropicalization of the curve $C$ with circuit $\plC_0$. Consider the associated family of tropical maps
\[
\varphi: \plC\to N_{\RR}
\]
If $\varphi$ does not contract the circuit, then let $\delta_f = 0$. Otherwise, let $\delta_f$ be the minimum distance from $\plC_0$ to a $1$-valent vertex of $\varphi^{-1}(\varphi(\plC_0))$. That is, the distance to the closest vertex supporting a flag that is not contracted by $\varphi$.  See Figure~\ref{fig: tropical-cricles} for an example. 

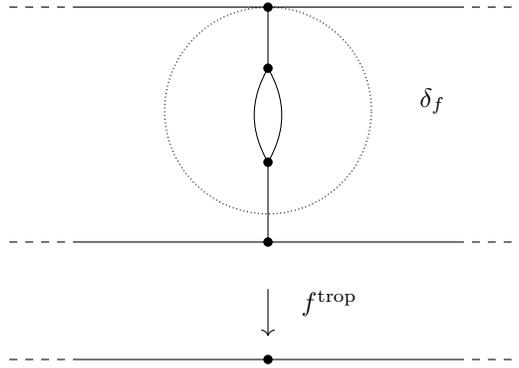
\begin{figure}[ht!]
\begin{tikzpicture}[scale=1.25]
\draw[dashed] (-2.75,0)--(-2,0);
\draw[dashed] (2,0)--(2.75,0);
\draw (-2,0)--(2,0);

\draw[dashed] (-2.75,-2.5)--(-2,-2.5);
\draw[dashed] (2,-2.5)--(2.75,-2.5);
\draw (-2,-2.5)--(2,-2.5);

\fill (0,0) circle (0.5mm);
\fill (0,-2.5) circle (0.5mm);

\draw (0,0)--(0,-0.65);
\path (0,-0.65) edge [bend left] (0,-1.65);
\path (0,-0.65) edge [bend right] (0,-1.65);
\draw (0,-1.65)--(0,-2.5);

\draw[dashed] (-2.75,-3.75)--(-2,-3.75);
\draw[dashed] (2,-3.75)--(2.75,-3.75);
\draw (-2,-3.75)--(2,-3.75);

\draw[->] (0,-3)--(0,-3.5);
\draw[densely dotted] (0,-1.1) circle (1.1);
\node at (1.75,-1) {\small $\delta_f$};
\node at (0.65,-3.15) {\small $f^\trop$};

\fill (0,-3.75) circle (0.5mm);
\fill (0,-0.65) circle (0.5mm);
\fill (0,-1.65) circle (0.5mm);

\end{tikzpicture}
\caption{A tropical map from a genus $1$ curve to $\Sigma_{\PP^1}$ contracting a circuit. The dotted circle corresponds to the circle whose radius is the minimal distance to a vertex supporting a non-contracted flag.}\label{fig: circle-around-circuit}\label{fig: tropical-cricles}
\end{figure}

These data will produce a rational bubbling of the source curve and a contraction thereof. First, subdivide $\plC$ such that every edge of $\varphi^{-1}(\varphi(\plC_0))$ terminates at a vertex; in Figure~\ref{fig: circle-around-circuit}, this amounts to introducing a vertex where the dotted circle crosses the lower vertical edge. This induces a logarithmic modification
\[
\tau: \widetilde C\to C.
\]
By the constructions of Section~\ref{sec: prelims-from-prequel}, we obtain a section comparable to the radii of $C$, and there is now an induced contraction 
\[
\gamma: \widetilde C\to \overnorm C,
\]
to a curve with a Gorenstein elliptic singularity. 

\begin{definition}
Keeping the notations above, an ordered logarithmic map $C\to Z$ as above is said to have the \textbf{factorization property} if the associated map $\widetilde C\to C\to Z$ factorizes as
\[
\begin{tikzcd}
\widetilde{C} \arrow{rr}{f}\arrow{rd} & & Z \\
& \overnorm{C}\arrow[swap]{ur}{\overnorm f}.
\end{tikzcd}
\]
\end{definition}

\begin{remark}
Note that it need not be the case that the map $\overnorm f$ is nonconstant on a branch of the elliptic singularity. This is because the map $\overnorm C\to Z$ may have highly degenerate contact order with the boundary of $Z$; it could be that the entire elliptic component is contracted. However, one may always replace $Z$ with a logarithmic modification $\mathscr Z$, i.e., a toric degeneration, such that the genus $1$ component maps to the dense torus of one of the components of $\mathscr Z$. In such an expansion, there will be at least one branch of the singularity along which the map is nonconstant. For example, if $Z = \PP^1$, it could be that the genus $1$ subcurve is contracted to one of the relative points $0$ or $\infty$ on $\PP^1$. In this case, one must first expand $\PP^1$ to a chain $\PP^1_{\mathrm{exp}}$, until the curve maps to the dense torus of a component in $\PP^1_{\mathrm{exp}}$. The choice of radius forces that the factorization is nonconstant on a branch of the singularity.
\end{remark}

Recall from Section~\ref{sec:to-toric} that a map $f : C \to Z$, where $Z$ is a toric variety with character lattice $N^\vee$, induces a homomorphism
\begin{equation*}
\alpha : N^\vee \to \Gamma(C, M_C^{\rm gp}) .
\end{equation*}
The factorization property depends only on $\alpha$ and not specifically on the morphism of toric varieties $C \to Z$.  For example, if $Z$ were a toric modification of another toric variety $Z'$, then the factorization properties for $C \to Z$ and $C \to Z'$ would coincide.  We offer a definition of the factorization property that makes this independence explicit.

\begin{definition} \label{def:factorization-no-toric}
Let $N$ and $N^\vee$ be finitely generated, free abelian groups, dual to each other, and let $C$ be a logarithmic curve over $S$.  Assume given $\alpha : N^\vee \to \Gamma(C, M_C^{\rm gp})$ and let $\overnorm\alpha$ be the induced morphism valued in $\Gamma(C, \overnorm M_C^{\rm gp})$.  Let $\plC_s$ be the tropicalization of $C_s$, for each geometric point $s$ of $S$.  Define $\delta_\alpha \in \Gamma(S, \overnorm M_S)$ fiberwise to be the largest $\lambda(v)$, among $v \in \plC$, such that $\overnorm\alpha$ is constant when viewed as a piecewise linear function on $\plC$.

Let $\upsilon : \widetilde C \to C$ and $\tau : \widetilde C \to \overnorm C$ be the destabilization and contraction constructed as above.  We say that $\alpha$ \textbf{satisfies the factorization property} if $\upsilon^\star \alpha$ descends along $\tau$ to $N^\vee \to \Gamma(\overnorm C, M_{\overnorm C}^{\rm gp})$.
\end{definition}

\begin{remark}
The factorization property is equivalent to requiring that $\widetilde C \to C \to N\otimes \mathbf G_m^{\mathrm{log}}$ factor through the contracted curve $\overnorm C$.
\end{remark}

\subsection{The stack of well-spaced logarithmic maps}\label{sec: well-spaced-logmaps} In this section, we construct a stack that we will later identify as the main component of the moduli space of genus~$1$ maps to a toric variety.  

We begin with some geometric motivation. Let $H$ be a subtorus of the dense torus $T$ of $Z$. After replacing $Z$ with a toric modification, there is a toric compactification $Z_H$ of the quotient torus $T/H$ and a toric morphism
\[
Z\to Z_H,
\]
extending the projection $T\to T/H$.

Let $f:C\to Z$ be a radial map over $S$, let $H$ a subtorus of the dense torus $T$, and assume that $Z\to Z_H$ exists for some $T/H$-toric variety $Z_H$. We say that $[f]$ satisfies the \textbf{factorization property for $H$} if the induced logarithmic map
\[
C\to Z\to Z_H
\]
satisfies the factorization property. 

This definition cannot be applied to an arbitrary toric variety $Z$ and an arbitrary subtorus $H \subset T$, since there may not be a toric map from $Z$ to an equivariant compactification of $T/H$.  For example, consider $Z = \PP^2$ and let $H$ be any $1$-dimensional subtorus. Since there is no non-constant map from $\PP^2$ to $\PP^1$, the assumption fails. 

There are two ways in which to overcome the issue. The first is to replace $Z$ with a logarithmic modification, which requires replacing $C$ with a logarithmic modification.  This logarithmic modification may not be defined over the base $S$, until we perform a logarithmic modification of $S$ as well~\cite[Proposition~4.5.2]{ACMW}.

It is conceptually simpler to use Definition~\ref{def:factorization-no-toric}, which does not require the map $Z \to Z_H$, but only the map of tori $T \to T/H$.  Indeed, let $N^\vee$ be the character lattice of $T$ and let $N^\vee_{T/H}$ be the character lattice of $T/H$.  Then the factorization property for $C \to Z \to Z_H$ is equivalent to the factorization property for the composition
\begin{equation*}
N^\vee_{T/H} \to N^\vee \to \Gamma(C, M_C^{\rm gp}) .
\end{equation*}
With this as motivation, we arrive at our definition:

\begin{definition} \label{def:fact-subtorus}
Let $f : C \to Z$ be a map from a radially aligned logarithmic curve to a toric variety $Z$ with dense torus $T$ and character lattice $N^\vee$.  Let $H$ be a subtorus of $T$ and let $N^\vee_{T/H}$ be the character lattice of $T/H$.  We say that $f$ \textbf{satisfies the factorization property} for $H$ if the map
\begin{equation*}
N^\vee_{T/H} \to N^\vee \to \Gamma(C, M_C^{\rm gp}) .
\end{equation*}
satisfies the factorization property of Definition~\ref{def:factorization-no-toric}.
\end{definition}

Geometrically, the condition is that $C \to N\otimes \mathbf{G}_{\log}\to N_{T/H}\otimes \mathbf{G}_{\log}$ should factor through $\overnorm C_H$, where $\overnorm C_H$ is constructed from $C \to N_{T/H} \otimes \mathbf{G}_{\log}$ as in Section~\ref{sec:factorization}. The equivalence between the formulations is a tautology: $\mathbf{G}_{\log}$ may simply be defined as the representing objects for global sections of $M_X^{\rm gp}$.

\begin{definition} \label{def:well-spaced}
Let $Z$ be a toric variety.
A radial logarithmic map $f:C\to Z$ is \textbf{well-spaced} if $f$ satisfies the factorization property for all subtori $H$ of $T$. 

Let $\mathcal W(Z)$ denote the category fibered in groupoids, over logarithmic schemes, of stable, well-spaced, radially aligned logarithmic stable maps to $Z$.
\end{definition}

Given a splitting $N_{T/H}^\vee \simeq \mathbf Z^r$, the factorization property  for $H$ is the conjunction of the factorization properties with respect to the tori dual to the direct summands of $\mathbf Z^r$.  When $N_{T/H}^\vee$ has rank~$1$, the factorization property asserts that a section of $M_{\tilde C}^{\rm gp}$ descends to a section of $M_{\overnorm C}^{\rm gp}$.  It is in this form that we will verify the algebraicity of the factorization property, below.



Let $C$ be a family of radially aligned logarithmic curves over $S$, assume that a section $\delta \in \Gamma(S, \overnorm M_S)$ is given, that $\tilde C \to C$ and $\tau : \tilde C \to \overnorm C$ are the associated semistable model and contraction of genus~$1$ component, respectively, and that $E \subset \tilde C$ is the exceptional locus of $\tau$.  We write $\pi : C \to S$ and $\overnorm\pi : \overnorm C \to S$ for the projections.  We fix a section $\alpha$ of $M_C^{\rm gp}$.  We will use $F$ for the subfunctor of the functor represented by $S$ on logarithmic schemes consisting of those $f : T \to S$ such that the restriction of $\alpha$ along $f$ has the factorization property --- in other words, the pullback of $\alpha$ to $\tilde C_T$ descends along $\tilde C_T \to \overnorm C_T$.  

In the following statements, we will say that the factorization property has a certain trait to mean that the functor $F$ has that trait, relative to the functor represented by $S$.  To unclutter the notation slightly, we will also assume without loss of generality, that $\tilde C \to C$ is an isomorphism, since doing so entails no loss of generality.

We begin by showing that the moduli functor of factorized maps is representable.

\begin{theorem} \label{thm:fact-rep}
The factorization property is representable by a (not necessarily strict) closed embedding of logarithmic schemes.
\end{theorem}

The proof proceeds as in Artin's criteria for algebraicity, although we do not need to invoke Artin's criteria directly because the morphism in question will turn out to be a closed embedding.  


The following two propositions refine \cite[Theorem~4.3]{RSW17A}, with essentially the same proof:

\begin{proposition} \label{prop:constructible}
The set of points of $S$ where $\alpha$ satisfies the factorization property is constructible.
\end{proposition}
\begin{proof}
We wish to show that the locus of points in $S$ where the section $\alpha$ of $M_C^{\rm gp}$ descends to $M_{\overnorm C}^{\rm gp}$ is constructible in $S$.

This assertion is local in the constructible topology on $S$, so we may assume that the dual graph of $C$ is locally constant over $S$.  The assertion is also local in the \'etale topology, so we can even assume the dual graph is constant.  Then there are two obstructions to descending $\alpha$ to a section of $M_{\overnorm C}^{\rm gp}$.  Since $M_{\overnorm C}^{\rm gp}$ is pulled back from $S$ near $\tau(E)$, the first obstruction is that $\overnorm\alpha$ should be constant on $E$.  But $\pi_\star \overnorm M_{C}^{\rm gp}$ is constant over $S$, so this holds on an open and closed subscheme of $S$.  In particular, this locus is quasicompact and constructible.  We may now restrict attention to this locus and assume the first obstruction vanishes.

Now $\alpha$ is a section of $\pi^\star M_S^{\rm gp} \subset M_C^{\rm gp}$ near $E$.  Since $\overnorm\alpha$ is constant on $E$ by assumption, we may, after localization in $S$, divide off a section pulled back from $S$, to ensure that $\overnorm\alpha(E) = 0$ and $\alpha$ is therefore a section of $\mathcal O_C^\star$.  We must show that the locus in $S$ where $\alpha$ is pulled back from $\mathcal O_{\overnorm C}^\star$ is constructible.

Regarding $\alpha$ instead as a section of $\mathcal O_C$, it is equivalent to show that the locus where $\alpha$ is pulled back from $\mathcal O_{\overnorm C}$ is constructible.  The rest of the proof is now the same as in \cite[Theorem~4.3]{RSW17A}.
\end{proof}

\begin{proposition}  \label{prop:proper}
The factorization property satisfies the valuative criterion for properness.
\end{proposition}
\begin{proof}
We assume that $S$ is the spectrum of a valuation ring, with $j : \eta \to S$ the inclusion of the generic point.  We wish to show that if $\alpha$ is a section of $M_C^{\rm gp}$ such that $j^\star \alpha$ is pulled back from a section of $M_{\overnorm C_\eta}$, then $\alpha$ is pulled back from a section of $M_{\overnorm C}$ over $\overnorm C$.

It suffices to assume that $S$ has the maximal logarithmic structure extending the logarithmic structure over $\eta$.  That is, $M_S = \mathcal O_S \mathop\times_{j_\star \mathcal O_\eta} j_\star M_\eta$.  This implies that $M_S^{\rm gp} = j_\star M_\eta^{\rm gp}$.

Our task is equivalent to showing that the square~\eqref{eqn:2} is cartesian:
\begin{equation} \label{eqn:2} \vcenter{\xymatrix{
M_{\overnorm C}^{\rm gp} \ar[r] \ar[d] & j_\star M_{\overnorm C_\eta}^{\rm gp} \ar[d] \\
\tau_\star M_C^{\rm gp} \ar[r] & j_\star \tau_\star M_{C_\eta}^{\rm gp}
}} \end{equation}

Away from $\tau(E)$, we know that $\varphi$ is an isomorphism, so it suffices to demonstrate the bijectivity of $\varphi$ near $\tau(E)$, where $M_{\overnorm C} = \overnorm\pi^\star M_S$.  Since the map $M_{\overnorm C}^{\rm gp} = \overnorm\pi^\star M_S^{\rm gp} \to \tau_\star M_C^{\rm gp}$ factors through $\tau_\star \pi^\star M_S^{\rm gp}$, it suffices to show that~\eqref{eqn:3} is cartesian:
\begin{equation} \label{eqn:3} \vcenter{\xymatrix{
\overnorm\pi^\star M_S^{\rm gp} \ar[r] \ar[d] & j_\star \overnorm\pi^\star M_S^{\rm gp} \ar[d] \\
\tau_\star \pi^\star M_S^{\rm gp} \ar[r] & j_\star \tau_\star \pi^\star M_S^{\rm gp}
}} \end{equation}
This reduces to showing that both of the following two squares are cartesian:
\begin{equation} \label{eqn:4} \vcenter{\hbox{\xymatrix{
\overnorm\pi^\star \overnorm M_S^{\rm gp} \ar[r] \ar[d] & j_\star \overnorm\pi^\star \overnorm M_S^{\rm gp} \ar[d] \\
\tau_\star \pi^\star \overnorm M_S^{\rm gp} \ar[r] & j_\star \tau_\star \pi^\star \overnorm M_S^{\rm gp}
} \qquad\qquad  \xymatrix{
\mathcal O_{\overnorm C}^\star \ar[r] \ar[d] & j_\star \mathcal O_{\overnorm C_\eta}^\star \ar[d] \\
\tau_\star \mathcal O_C^\star \ar[r] & j_\star \tau_\star \mathcal O_{C_\eta}^\star
}}} \end{equation}
We can check that the first square is cartesian on fibers, by proper base change for \'etale pushforward.  In that situation it is immediate, because $\pi^\star \overnorm M_S^{\rm gp}$ is constant on the fibers over $S$ and the fibers of $\tau$ are connected.  To see that the second square is cartesian, it is sufficient to see that~\eqref{eqn:5} is cartesian:
\begin{equation} \label{eqn:5} \vcenter{\xymatrix{
\mathcal O_{\overnorm C} \ar[r] \ar[d] & j_\star \mathcal O_{\overnorm C_\eta} \ar[d] \\
\tau_\star \mathcal O_C \ar[r] & j_\star \tau_\star \mathcal O_{C_\eta}
}} \end{equation}
The rest of the proof is exactly the same as the end of the proof of \cite[Theorem~4.3]{RSW17A}.
\end{proof}

Propositions~\ref{prop:constructible} and~\ref{prop:proper} combine to imply that the locus in $S$ where $\alpha \in M_C^{\rm gp}$ satisfies the factorization property is a closed subset of $S$.  To give that closed subset a scheme structure, we have two more propositions, the first of which requires the notion of a homogeneous functor, cf.\ \cite[Definition~2.5]{Rim}, \cite[Section~2]{obs}:

\begin{definition}
Let $F$ be a category fibered in groupoids over schemes.  We say $F$ is homogeneous if the map
\begin{equation*}
F(S') \to F(S) \mathop\times_{F(T)} F(T')
\end{equation*}
is an equivalence whenever $S' = S \mathop\amalg_T T'$ is the pushout of an infinitesimal extension $T \subset T'$ and an affine morphism $T \to S$.
\end{definition}

\begin{proposition} \label{prop:homogeneous}
The factorization property is homogeneous.
\end{proposition} 
\begin{proof}
Let $\tau : C' \to \overnorm C'$ be a contraction of genus~$1$ components over $S'$, where $S = T' \amalg_{T} S$ for a strict infinitesimal extension $T \subset T'$ and a strict affine morphism $\rho : T \to S$.  We must show that if $\alpha' \in \Gamma(C', M_{C'}^{\rm gp})$ and its restrictions $\beta'$ to $T'$ and $\alpha$ to $S$ satisfy the factorization property, then so does $\alpha'$.  This is in fact immediate in view of~\eqref{eqn:6}:
\begin{equation} \label{eqn:6}
M_{S'}^{\rm gp} = \rho_\star M_{T'}^{\rm gp} \mathop\times_{\rho_\star M_T^{\rm gp}} M_S^{\rm gp}
\end{equation}
Indeed, this assertion holds trivially on characteristic monoids, since $\overnorm M_{T'}^{\rm gp} \to \overnorm M_T^{\rm gp}$ is an isomorphism, so it comes down to the identification~\eqref{eqn:7},
\begin{equation} \label{eqn:7}
\mathcal O_{S'} = \rho_\star \mathcal O_{T'} \mathop\times_{\rho_\star \mathcal O_T} \mathcal O_S
\end{equation}
which is the definition of $S'$.
\end{proof}

\begin{proposition} \label{prop:integration}
The factorization property holds over a complete noetherian local ring if and only if it holds formally.
\end{proposition}
\begin{proof}
Here we must show that if $S$ is the spectrum of a complete noetherian local ring with maximal ideal $m$ and $S_i$ is the vanishing locus of $m^{i+1}$ then $\alpha \in \Gamma(C, M_S^{\rm gp})$ satisfies the factorization property if and only if its restriction to $S_i$ satisfies the factorization property for every $i$.  It is certainly the case that if $\overnorm\alpha \in \Gamma(C, \overnorm M_S^{\rm gp})$ is pulled back from $\Gamma(\overnorm C, \overnorm M_{\overnorm C}^{\rm gp})$ modulo every (in fact, \emph{any}) power of $m$ then so does $\overnorm\alpha$.  Indeed, this claim amounts to the assertion that if $\overnorm\alpha$ is constant on $E \cap \pi^{-1}(S_0)$ then $\overnorm\alpha$ is constant on $E$.  But $\overnorm M_C$ is an \'etale sheaf on $C$, so if $\overnorm\alpha$ vanishes at a point then it vanishes in an open neighborhood of that point.  Since the only open subset of $E$ containing $E \cap \pi^{-1}(S_0)$ is $E$ itself, we conclude that $\overnorm\alpha$ descends to $\overnorm C$ if $\overnorm\alpha \big|_{S_0}$ descends to $\overnorm\pi^{-1}(S_0)$.

Dividing $\alpha$ by a section of $M_S^{\rm gp}$, we can assume that $\overnorm\alpha = 0$ on $E$ and therefore that $\alpha$ is a section of $\mathcal O_C^\ast$ near $E$.  We can find an open subsets $\overnorm U \subset \overnorm C$ covering $\tau(E)$ such that $\overnorm\alpha$ vanishes on $U = \tau^{-1} \overnorm U$.  Then, by assumption, we we have $\tau_\star \alpha \big|_{\overnorm U} \in \mathcal O_{\overnorm U} / m^{i+1} \mathcal O_{\overnorm U} \subset \tau_\star(\mathcal O_U / m^{i+1} \mathcal O_U)$.  Passing to the limit, it follows that the image of $\tau_\star \alpha$ along $\tau_\star \mathcal O_U \to \varprojlim \tau_\star (\mathcal O_U / m^{i+1} \mathcal O_U)$ lies in $\varprojlim \mathcal O_{\overnorm U} / m^{i+1} \mathcal O_{\overnorm U} = \mathcal O_{\overnorm U}$.  But the theorem on formal functions guarantees that $\tau_\star \mathcal O_U \to \varprojlim \tau_\star(\mathcal O_U / m^{i+1} \mathcal O_U)$ is an isomorphism, so $\alpha$ has the factorization property, as required.
\end{proof}

As we will detail below, the propositions proved so far show that $F$ is representable by a closed subscheme of $S$ when restricted to the category of \emph{strict} logarithmic schemes over $S$.  To show that it is representable on the category of all logarithmic schemes over $S$, we require one more proposition.

\begin{proposition} \label{prop:minimal}
The factorization property has minimal monoids and the pullback of a minimal monoid is a minimal monoid.
\end{proposition}
\begin{proof}
Suppose that $\alpha$ satisfies the factorization property over $S_0$, where $S_0 \to S$ is a morphism of logarithmic schemes that is an isomorphism on the underlying schemes.  Write $C_0 = C \mathop\times_S S_0$ and let $\alpha_0$ be the image of $\alpha$ in $M_{C_0}^{\rm gp}$, which satisfies the factorization property by assumption.  Working locally, we can assume that $S$ is atomic and that the dual graph of $C$ is constant on the closed stratum.  Let $\plC$ be that dual graph.  Adjusting $\alpha$ by a section of $M_S^{\rm gp}$, we can assume that $\overnorm\alpha_0$ vanishes on the components contracted by $\tau$.  Let $Q$ be the set of all elements of $\overnorm M_S^{\rm gp}$ that arise as $\overnorm\alpha(v)$ as $v$ ranges over the vertices of $\plC$ contracted by $\tau$.  Then $Q$ is contained in the kernel of $\overnorm M_S^{\rm gp} \to \overnorm M_{S_0}^{\rm gp}$.  The minimal characteristic monoid on which we have the factorization property is therefore the saturation $\overnorm N_S$ of the image of $\overnorm M_S$ in the quotient of $\overnorm M_S^{\rm gp}$ by the subgroup generated by $Q$.  Since we have a map $\overnorm N_S \to \overnorm M_{S_0}$, we can pull back the logarithmic structure of $S_0$ to get a logarithmic structure $N_S$ on $S$ on which $\alpha$ has the factorization property.  It is immediate from the construction that $\overnorm N_S$ is minimal and that the construction commutes with pullback.
\end{proof}

\begin{proof}[Proof of Theorem~\ref{thm:fact-rep}]
Let $\tau : C \to \overnorm C$ be a contraction of genus~$1$ components over $S$ and let $\alpha$ be a section of $M_C^{\rm gp}$.  We wish to show that there is a closed subscheme $S' \subset S$ such that the pullback of $\alpha$ along $f : T \to S$ has the factorization property if and only if $f$ factors through $S'$.  We may assume that $S$ is of finite type, since all of the data in play are locally of finite presentation.

By Proposition~\ref{prop:constructible} and Proposition~\ref{prop:proper}, the subset of points of $S$ where $\alpha$ satisfies the factorization property is a closed subset $S_0$ of $S$.  If $S_0$ is given the reduced subscheme structure then $\alpha$ has the factorization property over $S_0$.  By Proposition~\ref{prop:homogeneous}, the subscheme structures on $S_0$ over which $\alpha$ has the factorization property are filtered.  Taking the limit of these closed subschemes yields a scheme $S'$, the spectrum of a complete noetherian local ring, such that the factorization property holds formally over $S'$.  But the underlying scheme of $S'$ must be the same as that of $S_0$, so the ideal of $S_0$ in $S'$ is nilpotent.  Thus $S'$ is actually an infinitesimal neighborhood of $S_0$ in $S$ and $S'$ is the maximal closed subscheme of $S$ over which the factorization property holds.

Now suppose that $f : T \to S$ is a strict morphism and that $\alpha$ has the factorization property over $T$.  We wish to show $f$ factors through $S'$.  We may assume $T$ is of finite type.  Certainly $f(T) \subset S_0$ as a set, so if $T_0$ is the reduced subscheme structure on $T$ the $f \big|_{T_0}$ factors through $S_0$.  Then $T$ is an infinitesimal extension of $T_0$ so the pushout $S_1 = T \amalg_{T_0} S_0$ is an infinitesimal extension of $S_0$ and $\alpha$ has the factorization property over $S_1$ by Proposition~\ref{prop:integration}.  It follows that $S_1 \subset S'$ and therefore $f$ factors through $S'$ as required.

By Proposition~\ref{prop:minimal} the factorization property has minimal monoids, so by Theorem~\ref{thm:gillam}, $S'$ with its minimal logarithmic structure represents the factorization property.
\end{proof}

\begin{theorem} \label{thm:w-rep-proper}
The category $\mathcal W(Z)$ is representable by a logarithmic algebraic stack.  After fixing the contact orders $\Gamma$, the substack $\mathcal W_\Gamma(Z)$ of maps with those contact orders is proper.
\end{theorem}

\begin{proof}
We have just seen that the factorization property is is representable by closed (hence proper) morphisms.  Since stability is an open condition, this shows $\mathcal W(Z)$ is a locally closed substack of $\mathfrak W(Z)$.  It also shows that $\mathcal W(Z)$ is a closed substack of the space $\Mbar^{\rm rad}_{1,n}(Z)$ of stable logarithmic maps from radially aligned curves to $Z$; this is a logarithmic modification of $\Mbar_{1,n}(Z)$, the space of stable logarithmic maps to $Z$, and is therefore proper.  It follows that $\mathcal W(Z)$ is proper.
\end{proof}

\subsection{Logarithmic smoothness}\label{sec: log-smooth}
The logarithmic tangent bundle of a toric variety $Z$ is trivial, and is naturally identified with $N\otimes_\ZZ\mathscr O_Z$, where $N = N(T) = \Hom(\Gm, T)$ is the cocharacter lattice of the dense torus.  Given a radial map $[f: C\to Z]$, the obstructions to deforming the map $[f]$ fixing the deformation of $[C]$ lie in the group 
\begin{equation*}
\operatorname{Obs}([f]) = H^1(C,f^\star T_Z^{\mathrm{log}}) = H^1(C,\mathscr O_C^{\dim Z}) 
\end{equation*}
with dimension 
\begin{equation*}
h^1(C,\mathscr O_C^{\dim Z}) = g(C)\cdot \dim Z.
\end{equation*}

Consider a torus quotient $T\to T/H$ and choose a compatible equivariant compactification
\[
Z\to Z_H,
\]
possibly passing from $Z$ to a modification, as in the previous section.
The quotient map induces a projection map on logarithmic tangent bundles, extending scalars from
\[
N(T)\to N(T/H)
\]
Choosing a splitting for the induced map on obstruction groups, we see that if the map $[\overnorm f: C\to \overnorm Z]$ is obstructed, then the map $[f]$ is also obstructed. The well-spacedness condition for radial logarithmic maps removes obstructions arising in this fashion. We now show that these obstructions are the only obstructions that arise. 

\begin{theorem} \label{thm:toric-log-smooth}
For any toric variety $Z$, the stack $\mathcal{W}(Z)$ is logarithmically smooth and unobstructed.
\end{theorem}

The proof will require the following lemma.

\begin{lemma} \label{lem:gen-pic}
Let $E$ be a connected Gorenstein curve of genus $1$ without genus~$0$ tails.  Let $E_\circ$ be the smooth locus of $E$.  The map $E_\circ \to \operatorname{Pic}^1(E)$ sending $x$ to $\mathcal O_E(x)$ is \'etale.
\end{lemma}

\begin{proof}
Consider the problem of deforming $x$, while fixing its image $\mathcal O_E(x)$ in the Picard group. The obstructions to these deformations lie in $H^1(E, \mathcal O_E(x))$.  Since $E$ has no genus~$0$ tails then $\omega_E$ is trivial, and Serre duality yields the requisite vanishing.

To see that the map has relative dimension $0$, note that the relative tangent space may be identified with the quotient of $H^0(E, \mathcal O_E(x))$ by $H^0(E, \mathcal O_E)$.  An application of Riemann--Roch shows that this quotient is trivial.
\end{proof}

\begin{corollary}  \label{cor:pic-vec}
Let $E$ be a connected Gorenstein curve of genus $1$ without genus~$0$ tails and let $a_1, \ldots, a_n \in \mathbf Z^m$.  Let $E_0$ be the smooth locus of $E$.  If the $a_i$ span $\mathbf Q^m$ then the map $E_0^n \to \operatorname{Pic}(E)^m$ sending $(x_1, \ldots, x_n)$ to the tuple of line bundles associated to the divisor with $\mathbf Z^m$-coefficients $\sum a_ix_i$.
\end{corollary}

\begin{proof}
By the elementary divisors theorem, we can assume that the $a_i$ are multiples of the standard basis vectors.  Since $E_0$ is smooth, we may project off the factors of $E_0^n$ where $a_i$ vanishes.  Then $E_0^m \to \operatorname{Pic}(E)^m$ is the product of the maps $E_0 \to \operatorname{Pic}^1(E) \xrightarrow{a_i} \operatorname{Pic}(E)$.  The maps $E_0 \to \operatorname{Pic}^1(E)$ are \'etale by Lemma~\ref{lem:gen-pic} and multiplication by $a_i$ is \'etale because we work in characteristic zero.
\end{proof}

\begin{proof}[Proof of Theorem~\ref{thm:toric-log-smooth}] 

We will use the logarithmic infinitesimal criterion for smoothness.  We must show that whenever $S \subset S'$ is a strict infinitesimal extension of logarithmic schemes, any morphism $S \to \mathcal W(Z)$ can be extended to $S'$, completing diagrams of the following form:
\begin{equation} \label{eqn:log-lift2} \vcenter{ \xymatrix{
S \ar[r] \ar[d] & \mathcal W(Z) \\
S' \ar@{-->}[ur]
}} \end{equation}

This assertion is local in $S$, so we can restrict to a neighborhood $U$ of a geometric point $s$, such that the map $\overline M^{\mathrm{gp}}_S|_U\to\overline M^{\mathrm{gp}}_{S,s}$ is an isomorphism.  Let $\plC$ be the tropicalization of $C_s$.  

\noindent
{\sc Filtering the deformations.}
Let $N$ and $N^\vee$ be the character and cocharacter lattices of $Z$, respectively.  The moduli map $S \to \mathcal W(Z)$ gives the data of a curve $C$ and a section $\alpha \in N \otimes \Gamma(C, M_C^{\rm gp})$.  For every torsion free quotient $N \to N'$ we obtain a map $\alpha_{N'} \in N' \otimes \Gamma(C, M_C^{\rm gp})$ by composition, which we also view as a map $C \to N' \otimes \mathbf G_{\log}$. For each such map $\plC\to N'\otimes \mathbf R$, there is a largest radius $\delta_{N'} \in \Gamma(S, \overnorm M_S)$ around the minimal circuit in $\plC$ whose interior is contracted by the map.

The radii $\delta_{N'}$ for varying $N'$ are totally ordered and necessarily finite in number. Rename these distinct radii $\delta_1 \geq \delta_2 \geq \cdots \geq \delta_k$.  

Since $\overnorm\alpha$ is a piecewise linear function on $\plC$ valued in $\overnorm M_S^{\rm gp}$ we have $\overnorm\alpha(v) - \overnorm\alpha(w) \in N \otimes \ell \subset N \otimes \overnorm M_S^{\rm gp}$ whenever $v$ and $w$ are connected by an edge of $\plC$ of length $\ell$.  We call $\frac{\overnorm\alpha(v) - \overnorm\alpha(w)}{\ell}$ the \emph{slope} of $\overnorm\alpha$ along that edge.

For each $i$, we take $N_i$ be the quotient of $N$ by the saturated sublattice spanned of the slopes of $\overnorm\alpha$ along the edges of $\plC$ contained inside the circle of radius $\delta_i$.  This gives an sequence of torsion free quotients $N \to N_k \to N_{k-1} \to \cdots \to N_1$.

For each $i$, let $Z_i = N_i \otimes \mathbf G_{\log}$.  Then we obtain a sequence of maps:
\begin{equation*}
\mathcal W(Z) \to \mathcal W(Z_k) \to \mathcal W(Z_{k-1}) \to \cdots \to \mathcal W(Z_1)
\end{equation*}
The first map is logarithmically \'etale since $Z$ is logarithmically \'etale over $\Hom(N, \mathbf G_{\log})$.  It will now suffice to show that $\mathcal W(Z_i) \to \mathcal W(Z_{i-1})$ is logarithmically smooth for all $i$.  For each $i$, let $\alpha_i$ be the image of $\alpha$ in $N_i \otimes \Gamma(C, M_C^{\rm gp})$.  We make the following observations:
\begin{enumerate}
\item We have $\delta_i \in \Gamma(S, \overnorm M_S)$ such that $\overnorm\alpha_i \in N_i \otimes \Gamma(C, \overnorm M_C^{\rm gp})$ is constant on the interior of the circle of radius $\delta_i$ around the central vertex of $\plC$;
\item the slopes of $\overnorm\beta$ on the edges of $\plC$ exiting the circle of radius $\delta_i$ span the kernel of $N_{i+1} \to N_i$ as a rational vector space.
\end{enumerate} 
The second observation requires a slight argument.  By definition, the slopes of $\overnorm\alpha$ within the circle of radius $\delta_i$ span the kernel of $N \to N_i$ rationally.  But if the edges inside the circle of radius $\delta_{i}$ together with those immediately exiting it spanned a smaller saturated subgroup of than the kernel of $N \to N_{i+1}$ then there would have been another $\delta_j$ in between $\delta_{i}$ and $\delta_{i+1}$.

\noindent
{\sc The iterative procedure.}
The map $\alpha : S \to \mathcal W(Z)$ gives families $C \leftarrow \widetilde C \to \overnorm C$ where $\overnorm C$ is the contraction of the circle of radius $\delta$.  We can regard $\alpha$ as an element of $N \otimes \Gamma(\overnorm C, M_{\overnorm C}^{\rm gp})$.  We examine extensions of these data to $C' \leftarrow \widetilde C' \to \overnorm C'$ and $\alpha' \in N \otimes M_{\overnorm C'}^{\rm gp}$.  The problem is addressed in two steps:  first we choose a deformation of $C$ (entailing deformations of $\widetilde C$ and $\overnorm C$), which is an obstructed problem, and then we try to lift $\alpha$, which can be obstructed for a fixed choice of $C'$.  We then revise our choice of deformation $C'$ to eliminate the obstruction to lifting $\alpha$.

The choices of $C'$ form a torsor under $H^1(C, T_{C/S}^{\log})$.  We will adjust $C'$ iteratively, lifting $\alpha_i$ to $\alpha'_i \in N_i \otimes \Gamma(\overnorm C', M_{\overnorm C'}^{\rm gp})$ based on an already selected lift of $\alpha_{i-1}$.  At each step, we will adjust $C'$ by a section of $H^1(C, T_{C/S}^{\log})$ that vanishes on the interior of the circle of radius $\delta_i$, thereby ensuring that our earlier choices are not broken by the later adjustments.

\noindent 
{\sc The obstruction group.}  Let $N'_i$ be the kernel of $N_i \to N_{i-1}$.  We indicate how $N'_i \otimes H^1(\overnorm C, \mathscr O_{\overnorm C}))$ functions as an obstruction group to deforming $\alpha_i$ once $\alpha_{i-1}$ and $C'$ are fixed.

By definition, the lifting problem~\eqref{eqn:log-lift2} is equivalent to the problem of extending $\alpha \in N_i \otimes \Gamma(\overnorm C, M_{\overnorm C}^{\rm gp})$ to $\alpha' \in N_i \otimes \Gamma(\overnorm C', M_{\overnorm C'}^{\rm gp})$. Recall that $\alpha$ gives an invertible sheaf, $\mathscr O_{\overnorm C}(-f(\overnorm\alpha))$ for each $f \in N_i^\vee$, and $f(\alpha)$ is a nowhere vanishing global section of $\mathscr O_{\overnorm C}(-f(\overnorm\alpha))$.  We will abuse notation slightly and think of $\mathscr O_{\overnorm C}(-\overnorm\alpha)$ as a family of invertible sheaves indexed by $N_i^\vee$ and $\alpha$ as a trivialization of this family.  Let $\overnorm\alpha'$ denote the unique extension of $\overnorm\alpha \in \Gamma(\overnorm C, \overnorm M_{\overnorm C}^{\rm gp})$ to $\overnorm M_{\overnorm C'}^{\rm gp}$.  Our task is to extend $\alpha$ to a trivialization of $\mathscr O_{\overnorm C'}(\overnorm\alpha')$.

If it exists, an extension will necessarily be a trivialization, so the obstruction to the existence of an extension is the isomorphism class of the deformation $\mathscr O_{\overnorm C'}(\overnorm\alpha')$, which lies in $N_i \otimes H^1(\overnorm C, \mathscr O_{\overnorm C})$.  By induction, the image of this obstruction in $N'_i \otimes H^1(\overnorm C, \mathscr O_{\overnorm C})$ vanishes, so our obstruction lies in $N'_i \otimes H^1(\overnorm C, \mathscr O_{\overnorm C})$.

\noindent
{\sc Deformations of the curve.} 
This obstruction may well be nonzero, but we are still free to choose $C'$.
The choice of $C'$ is a torsor under the deformation group $\operatorname{Def}(C) = H^1(C, T_{C/S}^{\log})$.  This gives a homomorphism
\begin{equation} \label{eqn:def-obs1}
H^1(C, T_{C/S}^{\log}) = \operatorname{Def}(C) \to \operatorname{Obs}_C(f) = N_i \otimes H^1(\overnorm C, \mathscr O_{\overnorm C})
\end{equation}
that we wish to show surjects onto $N'_i \otimes H^1(\overnorm C, \mathscr O_{\overnorm C})$.  Once this is done, we can modify $C'$ to eliminate the obstruction.

Since $C$ is a curve, the formation of $H^1(C, T^{\log}_{C/S})$ and of $H^1(\overnorm C, \mathcal O_{\overnorm C})$ commutes with base change.  By Nakayama's lemma, we may therefore demonstrate the surjectivity of~\eqref{eqn:def-obs1} by checking it on the fibers.  We may therefore replace $S$ with a geometric point and assume that $S$ is the spectrum of an algebraically closed field.

We let $\overnorm\plC$ be the dual graph of $\overnorm C$, where the interior of the circle of radius $\delta_i$ is treated as a single vertex.  For each vertex $v_i$ of $\overnorm\plC$ other than the central vertex $v_0$, let $e_j$ be the edge of $\overnorm\plC$ that is closest to $v_0$ (in other words, the edge on which $\lambda$ has negative slope when it is oriented away from $v_j$).  Let $p_j$ be the node of $\overnorm C$ corresponding to $e_j$.  For $j \neq 0$, the corresponding component of $\overnorm C$ is rational, and therefore $T_{C_j}^{\log}$ has nonnegative degree.  It follows that the maps
\begin{equation*}
H^0(C_j, T_{C_j}^{\log}) \to H^0(p_j, T_{C_j}^{\log}) \qquad \text{and} \qquad H^0(C_j, \mathcal O_{C_j}) \to H^0(p_j, \mathcal O_{C_j})
\end{equation*}
are surjective for all $j \neq 0$. From the normalization sequence, we see that there are decompositions:
\begin{equation*}
H^1(C, T_C^{\log}) = \bigoplus_j H^1(C_j, T_{C_j}^{\log}) \qquad \text{and} \qquad H^1(\overnorm C, \mathcal O_C) = \bigoplus_j H^1(C_j, \mathcal O_{C_j})
\end{equation*}
Since $H^1(C_j, \mathcal O_{C_j}) = 0$ for $j \neq 0$, it follows that we may reduce to the case $C = C_0$ and $\overnorm C = \overnorm C_0$.  Note that in this case, $\overnorm C_0$ has no genus~$0$ tails.

\noindent 
{\sc The obstruction class.} Now let $p_1, \ldots, p_n$ be the external points of $C$ corresponding to the edges $e_1, \ldots, e_n$ of $\overnorm\plC$ adjacent to the central vertex.  Then $H^1(C, T_{C/S}^{\log})$ contains a copy of $\sum_j T_{p_j/C}$ corresponding to deformation of $C$ as a logarithmic curve by moving the points $p_j$.  

Let $a_j \in N'_i$ be the slope of $\overnorm\alpha$ on $e_j$ and recall that the $a_j$ span $N'_i$ as a rational vector space.  Then the obstruction class in $N'_i \otimes H^1(\overnorm C, \mathcal O_{\overnorm C})$ is given by the following formula:
\begin{equation*}
\mathcal O_{\overnorm C}(\overnorm\alpha) = \mathcal O_{\overnorm C}(\sum a_j p_j) \in N'_i \otimes H^1(\mathcal O_{\overline C})
\end{equation*}
Thus the obstruction map
\begin{equation}\label{eqn: obs-class}
\sum_{j = 1}^n T_{p_j/C} \to H^1(C, T_{C/S}^{\log}) \to N_i \otimes H^1(\overnorm C, \mathcal O_{\overnorm C})
\end{equation}
restricts to the tangent map $\sum T_{p_j/C} \to N'_i \otimes H^1(\overnorm C, \mathscr O_{\overnorm C})$ considered in Corollary~\ref{cor:pic-vec}.  Since the $a_i$ span $N'_i$ rationally and $\overnorm C$ has no genus~$0$ tails, that corollary implies the desired surjectivity.  The points $p_j$ all lie on the boundary of the circle of radius $\delta_i$, so any class in $\sum T_{p_j/C} \subset H^1(C, T_{C/S}^{\log})$ vanishes on the interior of the circle of radius $\delta_i$, as required.
\end{proof}

\begin{remark}
The proof shows a stronger smoothness property, because we were able to cancel obstructions using only deformations of the marked points without needing to smooth any of the singularities of $\overnorm C$.
\end{remark}

It is a consequence of the proof that we can make following ``codimension $1$'' characterization of the moduli space.

\begin{proposition}\label{prop: reduction-of-dimension}
The substack of $\mathfrak W(Z)$ parametrizing stable radial maps that satisfy the factorization property for all subtori of codimension $1$ coincides with the space $\mathcal W(Z)$ of well-spaced radial maps.
\end{proposition}

\begin{proof}
It suffices to treat the case of $Z = \GG_{\log}^n$. Assume that a map $C\to \GG_{\log}^n$ fails to satisfy the factorization property for some subtorus $H$ in $T$. In keeping with the notation of the previous proof, we let $N'_H\subset N$ be the associated cocharacter subspace, and let $N_H$ be the quotient of $N$ by $N'_H$. We demonstrate that the map fails to satisfy the factorization property for a codimension $1$ subtorus. 

If the map is obstructed, then the cokernel in Equation~(\ref{eqn: obs-class}) is nonzero. It follows from Corollary~\ref{cor:pic-vec} that for the radius $\delta$ associated to the map $\plC\to N_H$, the exiting edge directions $a_i$ of $\plC\to N\otimes \RR$ at this radius do not span $N'_H$ rationally. The obstruction class described in the proof above therefore lies in the cokernel
\[
N_H/\mathrm{Span}(a_i)\otimes H^1(\overline C,\mathcal O_{\overline C}).
\]
Given any such obstruction class, we may find a projection by a character $N_H/\mathrm{Span}(a_i)\to \RR$, such that the class remains nonzero, by projecting onto the $1$-dimensional span of the obstruction class. This gives rise to a composition $N\to N_H\to \mathbf R$, and an associated map on tori $\GG_{\log}^n\to \GG_{\log}$. The induced map on logarithmic tangent bundle is given by extending scalars from the projection $N\to \RR$. The obstruction to lifting the map to $\GG_{\log}$ is the image of the obstruction to lifting the map to $\GG_{\log}^n$, under the projection
\[
N_H\otimes H^1(\overline C,\mathcal O_{\overline C})\to \RR \otimes H^1(\overline C,\mathcal O_{\overline C}).
\]
The resulting radial map $C\to \GG_{\log}$ therefore does not factorize, as the obstruction class is nonzero by construction. 
\end{proof}

\section{Realizability for genus one tropical curves}\label{sec: tropical-moduli}

In this section, we use the geometry of the moduli spaces $\mathcal W(Z)$ constructed in Section~\ref{sec:log-maps-to-toric} to resolve the tropical realizability problem in genus $1$. The results of this section give a precise description of the boundary complex of $\mathcal W(Z)$. As a consequence of the smoothness and properness of $\mathcal W(Z)$, tropical realizability reduces to a \emph{pointwise} calculation: we examine the unique non-topological condition characterizing the descent of a function from the normalization of a genus~$1$ singularity, and interpret it tropically as the realizability condition.

\subsection{Moduli of tropical maps} Fix a pair of dual lattices $N$ and $N^\vee$ of rank $r$ and a complete fan $\Sigma$ in the vector space $N_\RR$. 

\begin{definition}
The \textbf{combinatorial type} of a tropical stable map $[\plC\xrightarrow{f} \Sigma]$ consists of
\begin{enumerate}
\item The finite graph model $G$ underlying $\plC$.
\item For each vertex $v\in G$, the cone $\sigma_v\in \Sigma$ containing the image of $v$.
\item For each edge $e$, the slope $w_e$ and the primitive vector $u_e$ of $f(e)$. 
\end{enumerate}
\end{definition}

For tropical maps, the discrete data can be captured by the ``least generic'' map, defined below.

\begin{definition}
The \textbf{recession type} of a combinatorial type $\Theta$ is obtained from $[\plC\to \Sigma]$ by collapsing all bounded edges of $\plC$ to a single vertex, retaining the contact orders on the unbounded edges.  
\end{definition}

As explained in~\cite[Section 2]{R16}, once one fixes the recession type, there are finitely many combinatorial types of tropical stable maps with this recession type. This boundedness of combinatorial types is the essential content of~\cite[Section~3.1]{GS13}.

Given a type $\Theta$, there is a polyhedral cone $\sigma_\Theta$, whose relative interior parameterizes tropical stable maps with a fixed combinatorial type. This cone serves as a deformation space for maps of type $\Theta$. In~\cite[Section 2.2]{R16}, a generalized cone complex $T_\Gamma(\Sigma)$ is constructed, by taking a colimit of the cones above over a natural gluing operation. This is a coarse moduli space for maps of fixed recession type. 

\begin{remark}{\it (A moduli stack of tropical maps).} It is possible to promote this construction to a fine moduli stack of tropical maps.   By replacing the real edge lengths in $\plC$ with monoid-valued edge lengths, one obtains an appropriate notion of a family of tropical stable maps over a cone $\sigma$.  With this notion of family, the framework in \cite{CCUW} produces a cone stack $\mathcal{T}_\Gamma(\Sigma)$, with well-defined evaluation morphisms. The addition of a marked point with trivial contact order functions as a universal curve in this context. We avoid further discussion of this for two reasons. First, we will not need the stacks directly in this work, and can make do with the less conceptually natural, but more concrete generalized cone complex. Second, and more importantly, the precise relationship between the analytification of the moduli space of maps --- which coincides with the analytification of the coarse moduli space --- remains unclear at present.
\end{remark}

\subsection{Traditional tropicalization {\it \&} realizability} The tropicalization procedure discussed in the early parts of the paper uses the logarithmic structure, and differs from the one involving non-archimedean geometry. Accounting for the difference is the \textbf{tropical realizability problem}, and is the focus of this final section.

Let $K$ be a non-archimedean field extending $\CC$, where the latter is equipped with the trivial valuation. Let $Y$ be a $K$-scheme or stack, locally of finite type. The \textbf{Berkovich analytification} $Y^{\an}$ is a locally compact, Hausdorff topological space whose points are naturally identified with equivalence classes of pairs
$(L,y)$
where $L$ is a valued field extension of $K$ and $y$ is an $L$-valued point of $Y$. The equivalence is the one generated by identifying two such triples $(L,y)\sim (L',y')$ whenever there is an embedding of valued extensions $L\hookrightarrow L'$ sending $y$ to $y'$. See~\cite{Ber90,U14b,Yu14a} for Berkovich spaces and stacks and~\cite{ACMUW} for an introduction to analytic spaces in the context of logarithmic geometry.

Given a torus $\mathbf G_m^r = \spec(K[N^\vee])$, the \textbf{tropicalization} map is the continuous map
\[
\trop: \mathbf G_{m,\an}^n\to N\otimes \RR,
\]
that associates to an $L$-valued point of $\mathbf G_m^n$, its coordinatewise valuation. The tropicalization of a subvariety is defined by restriction.

Let $C\to \mathbf G_m^n$ be a map to a torus from a smooth curve of genus $g$. There is a natural factorization of topological spaces
\[
\begin{tikzcd}
C^{\an} \arrow{d}\arrow[end anchor={[xshift=4.3em]west}]{r} & \hskip4.3em \mathbf G_{m,\an}^r \arrow[d,shift left=2.5em] \\
\plC \arrow{r} & \trop(C^{\an})\subset \RR^r,
\end{tikzcd}
\]
The left vertical map is a deformation retraction onto a \textbf{skeleton}; see~\cite{BPR16} for details. There are at least two natural ways to extract the tropical curve $\plC$ from $[C\to \mathbf G_m^r]$.

\subsubsection{Abstract stable reduction} 

After choosing coordinates on the target, the map $[\varphi: C\to \mathbf G_m^n]$ is given by $n$ invertible functions on $C$. Let $\widehat C$ be the smooth projective model for $C$, and $q_1,\ldots, q_n$ the points at which these invertible functions acquire zeros or poles. If the map $\varphi$ is nonconstant, the pair $(\widehat C,q_1,\ldots, q_n)$ has negative Euler characteristic and thus admits a minimal model $\mathscr C\to \spec(R)$ over the valuation ring of $K$. Take the underlying graph of $\plC$ to be the dual graph of the special fiber of $\mathscr C$. Given an edge $e$ of $\plC$, the corresponding node $q_e$ of $\mathscr C$ has a local equation
\[
xy = \omega, \ \ \omega\in R.
\]
Set the length $\ell(e)$ equal to the valuation of the parameter $\omega$. 

\subsubsection{Universal property of minimality} Let $\widehat C\supset C$ be the projective model of $C$ with boundary $\partial \widehat C = \{q_1,\ldots, q_n\}$ and choose a toric compactification $Z$ of $\mathbf G_m^n$ such that the morphism
\[
(\widehat C,\partial \widehat C) \to (Z,\partial Z).
\] 
is a logarithmic map.  Letting $\mathscr L(Z)$ be the space of logarithmic stable maps to $Z$, this gives rise to a moduli map $\spec(K)\to \mathscr L(Z)$,
which, after a base change, extends to a map
\[
\spec(R)\to \mathscr L(Z),
\]
from the valuation ring. Let $k$ denote the residue field and $\Gamma$ the value group. Consider the logarithmic map
\[
\spec(\Gamma\to k)\to \mathscr L(Z),
\]
from the closed point, endowed with the (not necessarily coherent) logarithmic structure from the value group. By the universal property of minimality, this induces a factorization
\[
\spec(\Gamma\to k)\to \spec(P^{\mathrm{min}}\to k),
\]
where $P^{\mathrm{min}}$ is the stalk of the minimal monoid of $\mathscr L(Z)$ at the image of the closed point. We obtain a point of the dual cone $\Hom(P^{\mathrm{min}},\Gamma)$, which, as was previously discussed, is identified with a point in the cone of tropical maps of a fixed combinatorial type. See~\cite[Section 2]{R16} for details. 

\subsection{Expected dimension {\it \&} superabundance} Every tropical stable map $[f]$ of combinatorial type $\Theta$ has a deformation space, the moduli cone $\sigma_\Theta$. Superabundance is the phenomenon wherein this deformation space is larger than expected.

The \textit{overvalence} of a type $\Theta$ with underlying graph $G$ is defined as
\[
\mathrm{ov}(\Theta) = \sum_{p\in G:\val(p)\geq 4} \val(p)-3.
\]
The overvalence allows us to determine an expected topological dimension of the tropical deformation space as:
\[
\expdim \sigma_\Theta = (\dim(\Sigma)-3)(1-b_1(G))+n-\mathrm{ov}(\Theta),
\]
where $b_1(\plC)$ is the first Betti number of $G$. The actual dimension of $\sigma_\Theta$ cannot be less than the expected dimension, but may exceed it. For further details, see~\cite{Mi03,NS06,R16}.

\begin{definition}
A combinatorial type $\Theta$ is \textbf{superabundant} if the dimension of $\sigma_\Theta$ is strictly larger than the expected dimension.
\end{definition}

\setcounter{subsubsection}{\value{theorem}}
\subsubsection{Superabundance as tropical obstructedness} The deformation space of a map $[\varphi: C\to \PP^r]$ can be larger than expected because deformations can be obstructed. The dimension of the deformation space can be estimated using Riemann--Roch and the tangent-obstruction complex~\cite[Section 24.4]{Hori03}. One examines the restrictions on the complex structure of the curve that are forced by the map. In some cases, such as when $\varphi$ multiple covers its image or contracts a component, there are fewer such restrictions than expected. 

The situation in tropical geometry is similar. Given a tropical stable map $[f: \plC\to \Sigma]$ and a cycle of edges in $\plC$, the piecewise linearity of $f$ imposes restrictions on the edge lengths of this cycle. In particular, the edge lengths of a cycle are constrained by the condition that the total displacement around each cycle must vanish. If $\dim \Sigma = r$ the map is expected to impose $r$ conditions on the edge lengths of $\plC$ for each cycle, and the conditions imposed by different cycles are expected to be independent. However, if cycles are mapped to linear subspaces, or contracted altogether, there are fewer than the expected number of restrictions. 

In genus $1$, superabundance can be stated in a simplified form. In the following proposition, and the rest of the section, it will sometimes be convenient to forget the precise fan structure of $\Sigma$ and consider the map of metric spaces $\plC\to N_\RR$.

\setcounter{theorem}{\value{subsubsection}}
\begin{proposition}
Let $f:\plC\to \Sigma$ be a tropical map from a tropical curve of genus $1$. Then, $f$ is superabundant if and only if the image of the circuit $\plC_0$ is contained in a proper affine subspace of $\Sigma$. Equivalently, $f$ is superabundant if and only if there exists a character $\chi:N_\RR\to \RR$ such that the circuit $\plC_0$ is contracted under the composition
\[
\plC\to N_\RR\xrightarrow{\chi} \RR
\]
\end{proposition}

\begin{proof}
The first formulation is well known~\cite{KatLift,R16,Sp07}. For the second, choose a hyperplane containing the circuit and quotient by it.
\end{proof}

\begin{figure}
\begin{tikzpicture}
 \foreach \b in {-1,0,1,2,3}
  \fill (0,1) circle (0.05 cm);
    \fill (0,0) circle (0.05 cm);     \fill (1,0) circle (0.05 cm);     \fill (1,2) circle (0.05 cm);\fill (2,1) circle (0.05 cm);\fill (2,2) circle (0.05 cm);

\draw (0,0) -- (1,0);
\draw (0,0) -- (0,1);
\draw (0,1) -- (1,2);
\draw (1,2) -- (2,2);
\draw (1,0) -- (2,1);
\draw (2,1) -- (2,2);
\draw (-1,-1) -- (0,0);
\draw (-1,1) -- (0,1);
\draw (-1,1)--(-2,0);
\draw (-1,1)--(-1,2);
\draw[dashed] (-2,0)--(-2.5,-0.5);
\draw[dashed] (-1,2)--(-1,3);
\draw (1,-1) -- (1,0);
\draw (1,-1)--(0,-2);
\draw (1,-1)--(2,-1);
\draw[dashed] (0,-2)--(-0.5,-2.5);
\draw[dashed] (2,-1)--(3,-1);

\draw (1,2) -- (1,3);
\draw[dashed] (1,3) -- (1,3.5);
\draw (2,1) -- (3,1);
\draw[dashed] (3,1) -- (3.5,1);
\draw (2,2) -- (2.5,2.5);
\draw (2.5,2.5)--(2.5,3);
\draw (2.5,2.5)--(3,2.5);
\draw[dashed] (3,2.5) -- (3.5,2.5);
\draw[dashed] (2.5,3) -- (2.5,3.5);
\draw[dashed] (-1,-1) -- (-1.5,-1.5);
\end{tikzpicture}
\caption{A tropical genus $1$ curve in $\Sigma_{\PP^2}$ of degree $3$ with transverse contact orders. The curve is non-superabundant, as the edge directions of the circuit span $\RR^2$.}
\end{figure}
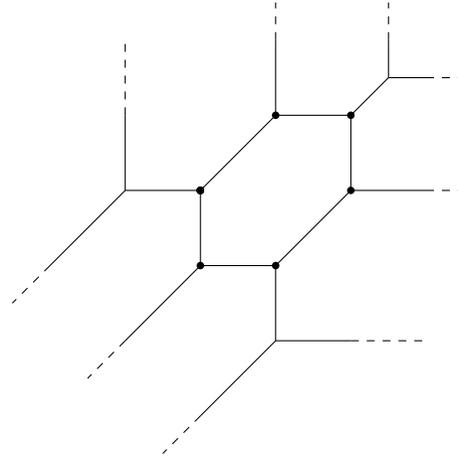

\subsection{Tropical realizability {\it \&} well-spacedness}\label{sec: well-spacedness} 

 The tropical realizability problem is as follows. 

\begin{question}
Given a tropical stable map $f: \plC\to N_\RR$, does there exist a smooth curve $C$ over a non-archimedean field $K$ and a map
\[
\varphi : C\to \mathbf G_m^r,
\]
such that $\varphi^\trop = f$?
\end{question}

\noindent
Such a tropical map is said to be \textbf{realizable}. Superabundance is intimately related to realizability, as the following result shows. For proofs, see~\cite{CFPU,R16,Sp07}.

\begin{theorem}
Let $f: \plC\to \Sigma$ be a tropical stable map of genus $1$ and combinatorial type $\Theta$. If $\plC$ has a vertex $v$ of genus $1$, then assume that the local map
\[
\mathrm{Star}(v)\to\Sigma
\]
is realizable. If the combinatorial type $\Theta$ is non-superabundant, then $f$ is realizable.
\end{theorem}

When a combinatorial type $\Theta$ is superabundant, there are additional constraints that are required to characterize the realizable locus.

A \textit{flag} of a tropical curve $\plC$ is a vertex $v$ together with a choice of tangent direction along an edge incident to $v$. The vertex $v$ will be referred to as the \textit{base} of the flag. Given a piecewise-linear function $f$ on a tropical curve $\plC$, we may speak of the \textbf{slope} of $f$ along a flag. 

\begin{definition}\label{def: well-spacedness}
Let $\plC$ be a tropical curve and let $\plC_0$ be its circuit. Given a flag $t\in \plC$, let $d(t,\plC_0)$ be the distance from the circuit to the base of the flag. A tropical stable map 
\[
F: \plC\to \RR
\] 
of genus $1$ is \textbf{well-spaced} if one of the following two conditions are met:  either
\begin{enumerate}
\item no open neighborhood of the circuit of $\plC$ is contracted, or
\item if a neighborhood of the circuit is contracted, let $t_1,\ldots,t_k$ be the flags whose base is mapped to $F(\plC_0)$ but along which $F$ has nonzero slope. Then, the minimum of the distances $\{d(t_i,\plC_0)\}_{i=1}^k$ occurs at least three times.
\end{enumerate}
\end{definition}

Well-spacedness when the target is a general fan is formulated by considering projections to $\RR$. 

\begin{definition}
A tropical stable map $\plC\to \Sigma$ of genus $1$ is \textbf{well-spaced} if for each character
\[
\chi: N_\RR\to \RR,
\]
the induced map $\plC\to \RR$ is well-spaced. 
\end{definition}

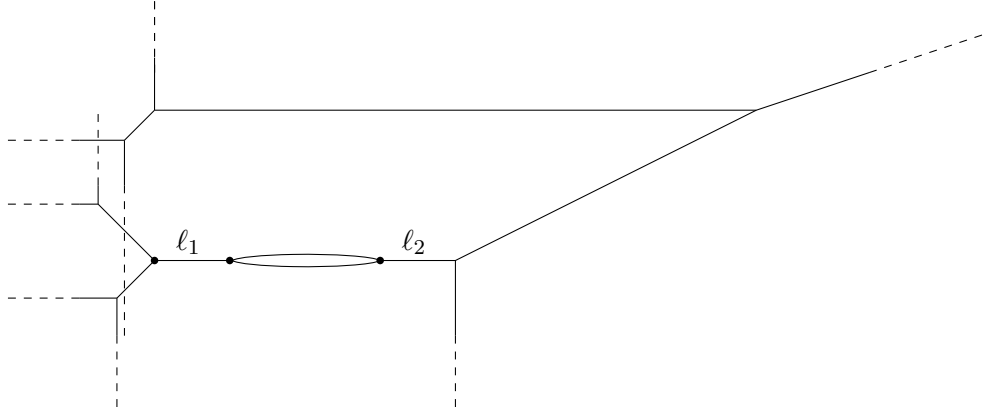
\begin{figure}[h!]
\begin{tikzpicture}[scale=1]
\draw (0,0) -- (0,1) -- (4,3) -- (5.5,3.5);
\draw[dashed] (5.5,3.5)--(7,4);
\draw[dashed] (-4,3.5)--(-4,4.5);
\draw[dashed] (-4,3.5)--(-4,4.5);
\draw[dashed] (-5,2.6)--(-6,2.6);
\draw[dashed] (-5,0.5)--(-6,0.5);
\draw[dashed] (-4.5,0)--(-4.5,-1);
\draw[dashed] (0,0)--(0,-1);

\draw (0,1) -- (-1,1);
\fill (-1,1) circle (0.5mm);
\fill (-3,1) circle (0.5mm);

\draw (-1,1) arc (10:170:1 and 0.1);
\draw (-1,1) arc (-10:-170:1 and 0.1);

\draw (-3,1) -- (-4,1);
\draw (-4,1) -- (-4.75,1.75);
\draw (-4,1)--(-4.5,0.5);
\draw (-4.5,0.5)--(-5,0.5);
\draw (-4.5,0.5)--(-4.5,0);

\draw (-4.75,1.75)--(-4.75,2);
\draw (-4.75,1.75)--(-5,1.75);
\draw[dashed] (-4.75,2)--(-4.75,3);
\draw[dashed] (-5,1.75)--(-6,1.75);

\draw (4,3) -- (-4,3);
\draw (-4,3)--(-4,3.7);
\draw (-4,3)--(-4.4,2.6);
\draw (-4.4,2.6)--(-5,2.6);
\draw (-4.4,2.6)--(-4.4,2);
\draw[dashed] (-4.4,2)--(-4.4,0);

\node at (-0.55,1.25) {$\ell_2$};
\node at (-3.55,1.25) {$\ell_1$};

\fill (-4,1) circle (0.5mm);

%
%
%
%
%
%
%

\end{tikzpicture}   
\caption{A superabundant tropical stable map to a Hirzebruch surface. The circuit is depicted to be flattened, indicating that its image is a line segment. Projection onto the vertical axis contracts the circuit. The curve is well-spaced if and only if $\ell_1 = \ell_2$.}
\end{figure}

\begin{warning}\label{warning: diff-defs}
The condition we call well-spacedness is strictly weaker condition than the one given originally by Speyer. In particular, the definition allows that the set of flags with nonzero $F$-slope $\{t_i\}$ can all be based at the same vertex. In Speyer's definition, there must be distinct vertices achieving this minimum. It has already been shown that Speyer's condition is not a necessary condition in the nontrivalent case~\cite[Theorem C]{R16}. The two definitions coincide when working with trivalent tropical curves whose vertex function is identically zero. To see this, observe that by the balancing condition, if a vertex supports one flag of nonzero $F$-slope. Thus, if two distinct vertices support flags with nonzero $F$-slope, then there are at least $4$ such flags. We will refer to this stronger condition as \textbf{Speyer's condition}; see Figure~\ref{fig: well-spaced-not-speyer}.
\end{warning}

\begin{remark}
We have chosen to state well-spacedness in terms of projections to $1$-dimensional vector spaces, as this is closest to the existing versions of the condition present in the literature. A reader who wishes to see the parallelism with Section~\ref{sec:factorization}  one could instead impose an appropriate condition on the quotient by any real subspace of $N_\RR$. 
\end{remark}

\begin{remark}
We make note of a consequence of this condition that is often useful in calculations. Let $\plC\to N_\RR$ be a tropical map from a genus $1$ curve. Let $L\subset N_\RR$ be the real span of the edge directions of the circuit of $\plC$. Let $\delta$ be the minimal radius around the circuit such that the edge directions inside the circle of radius $\delta$ span a subspace $L'$ strictly containing $L$. Let $m$ the difference in dimensions of $L$ and $L'$. Then if the tropical map is well-spaced, then at the circle of radius $\delta$, the curve $\plC$ exits the circle along least $m+2$ flags.
\end{remark}

This brings us to the main result of this section.

\begin{theorem}[Realizability of genus $1$ tropical curves]\label{thm: realizability}
Let $[\plC\to \Sigma]$ be a tropical stable map of genus $1$, and assume there is a minimal logarithmic map $[C\to Z]$ whose combinatorial type is that of $[\plC\to \Sigma]$. Then $[\plC\to \Sigma]$ is realizable if and only if it is well-spaced.
\end{theorem}

The proof will be completed in Section~\ref{sec: realizability-proof} after we establish some preliminaries in Section~\ref{sec:moduli-well-spaced}.


\begin{figure}[h!]
\begin{tikzpicture}[scale=1.25]
\draw[dashed] (-2.75,0)--(-2,0);
\draw[dashed] (2,0.5)--(2.75,.6875);
\draw[dashed] (2,-0.5)--(2.75,-0.6875);
\draw (-2,0)--(0,0);
\draw (0,0)--(2,0.5);
\draw (0,0)--(2,-0.5);

\draw[dashed] (-2.75,-2.5)--(-2,-2.5);
\draw[dashed] (2,-2.5)--(2.75,-2.5);
\draw (-2,-2.5)--(2,-2.5);

\fill (0,0) circle (0.5mm);
\fill (0,-2.5) circle (0.5mm);

\draw (0,0)--(0,-0.65);
\path (0,-0.65) edge [bend left] (0,-1.65);
\path (0,-0.65) edge [bend right] (0,-1.65);
\draw (0,-1.65)--(0,-2.5);
\fill (0,-0.65) circle (0.5mm);
\fill (0,-1.65) circle (0.5mm);

\draw[dashed] (-2.75,-3.75)--(-2,-3.75);
\draw[dashed] (2,-3.75)--(2.75,-3.75);
\draw (-2,-3.75)--(2,-3.75);

\draw[->] (0,-3)--(0,-3.5);
\fill (0,-3.75) circle (0.5mm);
\node at (-0.4,-0.4) {\small $\ell_1$};
\node at (-0.4,-2) {\small $\ell_2$};
\end{tikzpicture}
\caption{A tropical stable map that is well-spaced, but fails Speyer's condition. This map is well-spaced provided $\ell_1\leq\ell_2$, as there are three flags with nonzero slope based at the point of minimum distance to the circuit. Speyer's condition forces the equality $\ell_1 = \ell_2$.}\label{fig: well-spaced-not-speyer}
\end{figure}
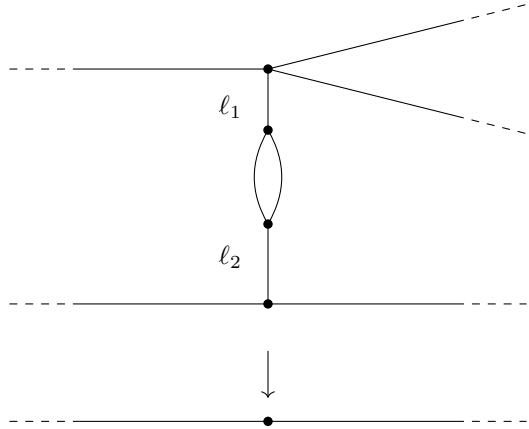

\subsection{Moduli of well-spaced tropical stable maps} \label{sec:moduli-well-spaced}
Let $T_\Gamma(\Sigma)$ be the moduli space of genus $1$ tropical stable maps with a fixed recession type $[\Gamma\to \Sigma]$. We abuse notation by understanding that the map to $\Sigma$ is part of the notation $\Gamma$. The well-spacedness condition commutes with automorphisms of tropical curves, and thus descends to a well-defined subset $W_\Gamma(\Sigma)$ of well-spaced tropical stable maps. We specify a subdivision of $T_\Gamma(\Sigma)$ such that $W_\Gamma(\Sigma)$ becomes an equidimensional subcomplex of the expected dimension. 

\begin{definition}
A \textbf{radially aligned} combinatorial type for a genus $1$ tropical stable map is a combinatorial type $\Theta$ for a tropical stable map, together with a choice of a reflexive and transitive binary relation $\preccurlyeq$ on the vertices such that, if $P$ is a path from the circuit of $G$ to a vertex $v$ that passes a vertex $v'$, then we have the relation
\[
v'\preccurlyeq v.
\]
Such a \textbf{radial combinatorial type} will be denoted $(\Theta,\preccurlyeq)$.
\end{definition}

\setcounter{subsubsection}{\value{theorem}}
\subsubsection{Constructing the tropical moduli space} Let $\widetilde W_\Gamma(\Sigma)$ denote the coarse moduli space of radially aligned tropical stable maps with recession type $\Gamma$. Given a radial combinatorial type $(\Theta,\preccurlyeq)$,  tropical maps of this type are parametrized by a face of a subdivision of the moduli cone $\sigma_\Theta$. There is a specialization relation among ordered combinatorial types: a cone $\sigma_{(\Theta',\preccurlyeq')}$ is a face of a cone $\sigma_{(\Theta,\preccurlyeq)}$ if and only if the following conditions hold.
\begin{enumerate}
\item Let $G$ and $G'$ be the underlying graphs of $\Theta$ and $\Theta'$ respectively. Then, $G'$ is obtained from $G$ by a (possibly trivial) sequence of edge contractions $\alpha: G\to G'$.
\item The edge contraction $G\to G'$ is order preserving:  if $v\preccurlyeq w$ then $\alpha(v)\preccurlyeq \alpha(w)$.
\item If $v'\in G'$ is a vertex with $\alpha(v') = v\in G$, then the cone $\sigma_{v'}$ is a face of $\sigma_v$.
\end{enumerate}

Let $W_\Gamma(\Sigma)$ be the subcomplex of $\widetilde W_\Gamma(\Sigma)$ parameterizing well-spaced radial tropical maps.

\setcounter{theorem}{\value{subsubsection}}
\begin{lemma}
The locus $W_\Gamma(\Sigma)$ is a subcomplex of $\widetilde W_\Gamma(\Sigma)$, and thus, is itself a generalized cone complex.
\end{lemma}

\begin{proof}
The well-spacedness condition can be described in terms of the equality of the vertices at minimum distance from the circuit, and thus form a cone of the generalized cone complex. The result follows immediately from this observation.
\end{proof}

\begin{remark}
A close relative of the space $W_\Gamma(\Sigma)$ appears in the thesis of Carolin Torchiani, namely the dense open set of $W_\Gamma(\Sigma)$ parametrizing curves with identically zero genus function. In particular, it is proved that this subcomplex is pure-dimensional of the expected dimension~\cite[Theorem 3.2.10]{Torch-thesis}. It follows from this that $W_\Gamma(\Sigma)$ is also pure-dimensional. In particular, we consider the following. It would be interesting to examine the fine structure of $W_\Gamma(\Sigma)$ further.  What can one say, for instance, about its homotopy type and connectivity properties?
\end{remark}

\subsection{Proof of Theorem~\ref{thm: realizability}}\label{sec: realizability-proof} 
We know from Section~\ref{sec: log-smooth} that the moduli space of well-spaced logarithmic stable maps $\mathcal W_\Gamma(Z)$ is proper and smooth.  By definition, it is the locus of stable maps in $\mathfrak W_\Gamma(Z)$ that satisfy the factorization property for every subtorus of the dense torus of $Z$.  Our task is to show that the logarithmic well-spacedness condition is equivalent to the tropical well-spacedness condition. By Proposition~\ref{prop: reduction-of-dimension}, the logarithmic well-spacedness condition is the conjunction of the factorization properties for all $1$-dimensional quotients of $Z$.  Since the tropical well-spacedness condition was formulated in terms of $1$-dimensional quotients, it suffices to check the equivalence for every subtorus $H$ in $Z$ of codimension $1$. Replacing $Z$ with a modification and passing to the quotient, our obligation reduces to checking that \emph{a tropical map $\plC \to \RR$, in which all vertices of $\plC$ have genus~$0$, is well-spaced if and only if it is the tropicalization of a radial map $C \to \PP^1$ satisfying the factorization property}.

Let $C$ be a logarithmic curve with tropicalization $\plC$.  The map $[\plC \to \RR]$ induces a destabilization $\upsilon : \widetilde C \to C$ and a contraction $\tau : \widetilde C \to \overnorm C$.  The map itself can be regarded as a section $\overnorm\alpha$ of $\overnorm M_C^{\rm gp}$.  This pulls back to $\overnorm M_{\widetilde C}^{\rm gp}$ and then descends to $\overnorm M_{\overnorm C}^{\rm gp}$, since it is constant on the components collapsed by $\tau$.  Adding a constant to $\overnorm\alpha$ does not change whether it is well-spaced in either the logarithmic or the tropical sense, so we assume that $\overnorm\alpha$ takes the value $0$ on the circuit component of $\overnorm C$.

We must show that $\overnorm\alpha$ lifts to a section $\alpha$ of $M_{\overnorm C}^{\rm gp}$ if and only if $\plC \to \RR$ is well-spaced.  Indeed, if $\alpha$ is a section of $M_{\overnorm C}^{\rm gp}$ then $\upsilon_\star \tau^\star \alpha$ is a section of $\upsilon_\star M_{\widetilde C}^{\rm gp} = M_C^{\rm gp}$ by~\cite[Appendix~B]{AMW12}, and gives a map $C \to \PP^1$ with the factorization property.

Let $E$ denote the circuit component of $\overnorm C$ and $E^\circ$ its interior, excluding the nodes where $E$ is joined to the rest of $\overnorm C$ (in other words, the locus in $E$ where the logarithmic structure is pulled back from the base).  Since $\overnorm\alpha(E) = 0$, the lift $\alpha \big|_{E^\circ}$, if it exists, will be in $\mathcal O_{E^\circ}^\star \subset M_{E^{\circ}}$.  Regarded as a rational function on $E$, this lift must have zeroes and poles along the points of attachment between $E$ and the rest of $\overnorm C$ as specified by the outgoing slopes of $\overnorm\alpha$ along the corresponding edges (see Section~\ref{sec:tropicalization}). Once $\alpha \big|_{E^\circ}$ has been found, there is no obstruction to extending it to all of $\overnorm C$, since the rest of the curve is a forest of rational curves and $\overnorm\alpha$ is balanced.  The following lemma determines whether $\alpha \big|_{E^\circ}$ can be found, and completes the proof of the theorem.

\begin{lemma}
Let $E$ be a Gorenstein, genus~$1$ curve with no nodes and $m$ branches, let $a_1, \ldots, a_n$ be nonzero integers, and let $P$ be a partition of $1, \ldots, n$ into $m$ parts.  Assume that, for each $p \in P$, we have $\sum_{i \in p} a_i = 0$.  If $n \geq 3$ then there is a configuration of distinct points $x_1, \ldots, x_n$ on $E$, with each point lying in the component corresponding to its part of the partition, such that $\mathcal O_E(\sum a_i x_i)$ is trivial.  If $n = 2$ then there is no such configuration.
\end{lemma}
\begin{proof}
Let $\nu : F \to E$ be the seminormalization and let $\omega_E$ be the dualizing sheaf.  For any configuration of the $x_i$, subject to the degree constraint in the statement, there is a rational function $f$ on $F$ with divisor $\sum a_i x_i$, and $f$ is unique up to scaling.  We wish to determine whether $f$ descends to~$E$.

Let $y \in F$ be the preimage of the singular point of $E$ and let $\phi$ be a nonzero global differential on $E$.  Let $F_j$ be the components of $F$ and let $\nu_j : F_j \to E$ be the restrictions of $\nu$ and let $f_j$ be the restriction of $f$ to $F_j$.  Let $t_j$ be a local parameter for $F_j$ at $y$ and let $b_j$ be the linear term of the expansion of $f_j$ in terms of $t_j$.  It was shown in Section~\ref{sec: genus-1-singularities} that there are nonzero constants $c_j$ such that $f$ descends to $E$ if and only if
\begin{equation} \label{eqn:descent-obstruction}
\sum_j c_j b_j = 0 .
\end{equation}
We argue that under the hypothesis of the Lemma it is possible to configure the $x_i$ on each component $F_j$ to make $b_j$ take any value we like.  Indeed, if we decide $f(y)$ should be $1$ then $f_j$ has the formula
\begin{equation*}
f_j = \prod_i (1 - x_i^{-1} t_j)^{a_i}
\end{equation*}
with the product taken over those $i$ such that $x_i$ lies on $F_j$.  The 
linear part is
\begin{equation*}
b_j = - \sum a_i x_i^{-1} .
\end{equation*}
By adjusting the positions of the $x_i$, we can arrange for $b_j$ to have any nonzero value we like.  If $p_j$ consists of at least $3$ points $x_i$ then it is possible to achieve any value for $b_j$, including $0$, but if $p_j$ consists of only two points, $x_{i}$ and $x_{i'}$ then $a_{i'} = -a_{i}$ and it is impossible for $b_j$ to take the value $0$.  

Thus we can solve~\eqref{eqn:descent-obstruction} provided either that there are at least two branches at $y$ or there is one branch containing at least $3$ of the $x_i$.  The one remaining case is where there is one branch containing $2$ points.  In that case, $c_1 \neq 0$ and the remarks in the last paragraph show there is no solution to~\eqref{eqn:descent-obstruction}.
%
%
%
\end{proof}

The above result determines the dual complex of the space $\mathcal W_\Gamma(Z)$, and we obtain the following as a consequence of general structural results about tropicalizations of logarithmic schemes. Let $\mathcal W^\circ_\Gamma(Z)$ denote the locus of maps with trivial logarithmic structure.

\begin{theorem} \label{thm:nonarch}
There is a continuous tropicalization map 
\[
\trop: \mathcal W^{\circ,\an}_{\Gamma}(Z) \to W_\Gamma(\Sigma),
\]
functorial with respect to evaluation moprhisms and forgetful morphisms to the moduli space of curves. Set theoretically, this map sends a family of logarithmic stable maps to its tropicalization. There is a factorization
\[
\begin{tikzcd}
\mathcal W^{\circ,\an}_{\Gamma}(Z) \arrow{rr}{\trop} \arrow[swap]{dr}{\bm p} & & W_\Gamma(\Sigma) \\
& \mathrm{P}_\Gamma(\Sigma) \arrow[swap]{ur}{\trop_{\mathfrak{S}}}, &
\end{tikzcd}
\]
where the map $\bm p$ is a deformation retraction onto a generalized cone complex, and admits a canonical continuous section. The map $\trop_{\mathfrak{S}}$ is finite and is an isomorphism of cones upon restriction to any face. 
\end{theorem}

\begin{proof}
With the identification of the tropical maps that arise as tropicalizations of one-parameter families, the proof of the result is a cosmetic variation on similar results in the literature~\cite{CMR14a,R15b,R16}. By Theorem~\ref{thm: realizability}, the tropicalization of any family of logarithmic stable maps over a valuation ring is well-spaced. Once this is established, the continuity, functoriality, and finiteness of $\trop_{\mathfrak S}$ follow from~\cite[Theorem 2.6.2]{R16} and the uniqueness of minimal morphisms of logarithmic schemes up to saturation~\cite{Wis16b}. The saturation index of a combinatorial type $(\Theta,\preccurlyeq)$ is equal to the cardinality of the fibers of $\trop_{\mathfrak S}$, as explained in~\cite{R15b,R16}. Since $\mathcal W_\Gamma(Z)$ is a toroidal compactification, the existence of a section from the skeleton follows from results of Thuillier~\cite{ACP,Thu07}. Compatbility with forgetful and evaluation morphisms follows from~\cite[Theorem 1.1]{U13}. 
\end{proof}

\bibliographystyle{siam} 
\bibliography{EllipticStableMaps}

\begin{thebibliography}{10}

\bibitem{ACP}
{\sc D.~Abramovich, L.~Caporaso, and S.~Payne}, {\em The tropicalization of the
  moduli space of curves}, Ann. Sci. {\'E}c. Norm. Sup{\'e}r., 48 (2015),
  pp.~765--809.

\bibitem{AC11}
{\sc D.~Abramovich and Q.~Chen}, {\em Stable logarithmic maps to
  {D}eligne-{F}altings pairs {II}}, Asian J. Math., 18 (2014), pp.~465--488.

\bibitem{ACGS15}
{\sc D.~Abramovich, Q.~Chen, M.~Gross, and B.~Siebert}, {\em {Decomposition of
  degenerate Gromov-Witten invariants}}, In preparation,  (2015).

\bibitem{ACMUW}
{\sc D.~Abramovich, Q.~Chen, S.~Marcus, M.~Ulirsch, and J.~Wise}, {\em
  Skeletons and fans of logarithmic structures}, in Nonarchimedean and Tropical
  Geometry, M.~Baker and S.~Payne, eds., Simons Symposia, Springer, 2016,
  pp.~287--336.

\bibitem{ACMW}
{\sc D.~Abramovich, Q.~Chen, S.~Marcus, and J.~Wise}, {\em Boundedness of the
  space of stable logarithmic maps}, J. Eur. Math. Soc. (to appear)
  arXiv:1408.0869,  (2014).

\bibitem{AMW12}
{\sc D.~Abramovich, S.~Marcus, and J.~Wise}, {\em {Comparison theorems for
  Gromov--Witten invariants of smooth pairs and of degenerations}}, in Ann.
  Inst. Fourier, vol.~64, 2014, pp.~1611--1667.

\bibitem{AW}
{\sc D.~Abramovich and J.~Wise}, {\em {Birational invariance in logarithmic
  Gromov-Witten theory}}, arXiv:1306.1222,  (2013).

\bibitem{AK}
{\sc A.~Altman and S.~Kleiman}, {\em Introduction to {G}rothendieck duality
  theory}, Lecture Notes in Mathematics, Vol. 146, Springer-Verlag, Berlin-New
  York, 1970.

\bibitem{BPR16}
{\sc M.~Baker, S.~Payne, and J.~Rabinoff}, {\em Nonarchimedean geometry,
  tropicalization, and metrics on curves}, Algebr. Geom., 3 (2016),
  pp.~63--105.

\bibitem{BNR19}
{\sc L.~Battistella, N.~Nabijou, and D.~Ranganathan}, {\em Curve counting in
  genus one: elliptic singularities and relative geometry}, arXiv:1907.00024,
  (2019).

\bibitem{Ber90}
{\sc V.~G. Berkovich}, {\em Spectral theory and analytic geometry over
  non-Archimedean fields}, vol.~33, American Mathematical Society, 1990.

\bibitem{BBM14}
{\sc B.~Bertrand, E.~Brugall{\'e}, and G.~Mikhalkin}, {\em Genus 0
  characteristic numbers of the tropical projective plane}, Comp. Math., 150
  (2014), pp.~46--104.

\bibitem{CCUW}
{\sc R.~Cavalieri, M.~Chan, M.~Ulirsch, and J.~Wise}, {\em A moduli stack of
  tropical curves}, arXiv preprint arXiv:1704.03806,  (2017).

\bibitem{CJM1}
{\sc R.~Cavalieri, P.~Johnson, and H.~Markwig}, {\em Tropical {H}urwitz
  numbers}, J. Algebraic Combin., 32 (2010), pp.~241--265.

\bibitem{CMR14a}
{\sc R.~Cavalieri, H.~Markwig, and D.~Ranganathan}, {\em {Tropicalizing the
  space of admissible covers}}, {Math. Ann.}, 364 (2016), pp.~1275--1313.

\bibitem{Che10}
{\sc Q.~Chen}, {\em Stable logarithmic maps to {D}eligne-{F}altings pairs {I}},
  Ann. of Math., 180 (2014), pp.~341--392.

\bibitem{CFPU}
{\sc M.-W. Cheung, L.~Fantini, J.~Park, and M.~Ulirsch}, {\em Faithful
  realizability of tropical curves}, Int. Math. Res. Not.,  (2015), p.~rnv269.

\bibitem{Gi12}
{\sc W.~Gillam}, {\em Logarithmic stacks and minimality}, Int. J. Math., 23
  (2012).

\bibitem{Gro14}
{\sc A.~Gross}, {\em Correspondence theorems via tropicalizations of moduli
  spaces}, arXiv:1401.4626,  (2014).

\bibitem{Gro15}
\leavevmode\vrule height 2pt depth -1.6pt width 23pt, {\em Intersection theory
  on tropicalizations of toroidal embeddings}, arXiv preprint arXiv:1510.04604,
   (2015).

\bibitem{GS13}
{\sc M.~Gross and B.~Siebert}, {\em {Logarithmic Gromov-Witten invariants}}, J.
  Amer. Math. Soc., 26 (2013), pp.~451--510.

\bibitem{HK12}
{\sc D.~Helm and E.~Katz}, {\em {Monodromy filtrations and the topology of
  tropical varieties}}, {Can. J. Math.}, 64 (2012), pp.~845--868.

\bibitem{Hori03}
{\sc K.~Hori, S.~Katz, A.~Klemm, R.~Pandharipande, R.~Thomas, C.~Vafa,
  R.~Vakil, and E.~Zaslow}, {\em Mirror symmetry}, vol.~1 of Clay Mathematics
  Monographs, American Mathematical Society, Providence, RI, 2003.
\newblock With a preface by Vafa.

\bibitem{JR17}
{\sc D.~Jensen and D.~Ranganathan}, {\em {Brill-Noether theory for curves of a
  fixed gonality}}, arXiv:1701.06579,  (2017).

\bibitem{Kat00}
{\sc F.~Kato}, {\em Log smooth deformation and moduli of log smooth curves},
  Int. J. Math., 11 (2000), pp.~215--232.

\bibitem{Kat89}
{\sc K.~Kato}, {\em {Logarithmic structures of Fontaine-Illusie}}, Algebraic
  analysis, geometry, and number theory (Baltimore, MD, 1988),  (1989),
  pp.~191--224.

\bibitem{KatLift}
{\sc E.~Katz}, {\em Lifting tropical curves in space and linear systems on
  graphs}, Adv. Math., 230 (2012), pp.~853--875.

\bibitem{LR15}
{\sc Y.~Len and D.~Ranganathan}, {\em Enumerative geometry of elliptic curves
  on toric surfaces}, To appear in Isr. Math. J. (arXiv:1510.08556),  (2017).

\bibitem{MR16}
{\sc T.~Mandel and H.~Ruddat}, {\em {Descendant log Gromov-Witten invariants
  for toric varieties and tropical curves}}, arXiv:1612.02402,  (2016).

\bibitem{Mi03}
{\sc G.~Mikhalkin}, {\em Enumerative tropical geometry in {${\mathbb{R}^2}$}},
  J. Amer. Math. Soc, 18 (2005), pp.~313--377.

\bibitem{logpic}
{\sc S.~Molcho and J.~Wise}, {\em {The logarithmic Picard group and its
  tropicalization}}, arXiv:1807.11364,  (2018).

\bibitem{Ni09}
{\sc T.~Nishinou}, {\em Correspondence theorems for tropical curves},
  arXiv:0912.5090,  (2009).

\bibitem{NS06}
{\sc T.~Nishinou and B.~Siebert}, {\em Toric degenerations of toric varieties
  and tropical curves}, Duke Math. J., 135 (2006), pp.~1--51.

\bibitem{Par11}
{\sc B.~Parker}, {\em {Gromov Witten invariants of exploded manifolds}},
  arXiv:1102.0158,  (2011).

\bibitem{R15b}
{\sc D.~Ranganathan}, {\em {Skeletons of stable maps I: rational curves in
  toric varieties}}, J. Lond. Math. Soc., 95 (2017), pp.~804--832.

\bibitem{R16}
\leavevmode\vrule height 2pt depth -1.6pt width 23pt, {\em {Skeletons of stable
  maps II: superabundant geometries}}, Res. Math. Sci., 4 (2017).

\bibitem{R15a}
\leavevmode\vrule height 2pt depth -1.6pt width 23pt, {\em {Superabundant
  curves and the Artin fan}}, Int. Math. Res. Not., 2017 (2017),
  pp.~1103--1115.

\bibitem{R19}
\leavevmode\vrule height 2pt depth -1.6pt width 23pt, {\em {Logarithmic
  Gromov-Witten theory with expansions}}, arXiv:1903.09006,  (2019).

\bibitem{RSW17A}
{\sc D.~Ranganathan, K.~Santos-Parker, and J.~Wise}, {\em Moduli of stable maps
  in genus one {\it \&} logarithmic geometry {I}}, arXiv:1708.02359 Geom. Top.
  (to appear),  (2019).

\bibitem{RW19}
{\sc D.~Ranganathan and J.~Wise}, {\em Rational curves in the logarithmic
  multiplicative group}, arXiv:1901.08489 Proc. Amer. Math. Soc. (to appear),
  (2019).

\bibitem{Rim}
{\sc D.~S. Rim}, {\em Formal deformation theory}, in Groupes de Monodromie en
  G{\'e}om{\'e}trie Alg{\'e}brique, Springer, 1972, pp.~32--132.

\bibitem{Smyth}
{\sc D.~I. {Smyth}}, {\em {Modular compactifications of the space of pointed
  elliptic curves. I.}}, Comp. Math., 147 (2011), pp.~877--913.

\bibitem{Sp-thesis}
{\sc D.~E. Speyer}, {\em Tropical geometry}, PhD thesis, University of
  California, Berkeley, 2005.

\bibitem{Sp07}
\leavevmode\vrule height 2pt depth -1.6pt width 23pt, {\em {Parameterizing
  tropical curves. I: Curves of genus zero and one.}}, {Algebra Number Theory},
  8 (2014), pp.~963--998.

\bibitem{Tev07}
{\sc J.~Tevelev}, {\em Compactifications of subvarieties of tori}, Amer. J.
  Math., 129 (2007), pp.~1087--1104.

\bibitem{Thu07}
{\sc A.~Thuillier}, {\em {G{\'e}om{\'e}trie toro{\"\i}dale et g{\'e}om{\'e}trie
  analytique non archim{\'e}dienne. Application au type d'homotopie de certains
  sch{\'e}mas formels}}, Manuscripta Math., 123 (2007), pp.~381--451.

\bibitem{Torch-thesis}
{\sc C.~Torchiani}, {\em Enumerative geometry of rational and elliptic tropical
  curves in $\mathbf R^m$}, PhD thesis, Technische Universit{\"a}t
  Kaiserslautern, 2014.

\bibitem{Tyo12}
{\sc I.~Tyomkin}, {\em Tropical geometry and correspondence theorems via toric
  stacks}, Math. Ann., 353 (2012), pp.~945--995.

\bibitem{U13}
{\sc M.~Ulirsch}, {\em Functorial tropicalization of logarithmic schemes: The
  case of constant coefficients}, arXiv:1310.6269,  (2013).

\bibitem{U14b}
\leavevmode\vrule height 2pt depth -1.6pt width 23pt, {\em Tropicalization is a
  non-archimedean analytic stack quotient}, Math. Res. Lett. arXiv:1410.2216,
  (To appear).

\bibitem{VZ08}
{\sc R.~Vakil and A.~Zinger}, {\em A desingularization of the main component of
  the moduli space of genus-one stable maps into $\mathbb{P}^n$}, Geom. Top.,
  12 (2008), pp.~1--95.

\bibitem{obs}
{\sc J.~Wise}, {\em Obstruction theories and virtual fundamental classes},
  arXiv preprint arXiv:1111.4200,  (2011).

\bibitem{Wis16a}
\leavevmode\vrule height 2pt depth -1.6pt width 23pt, {\em Moduli of morphisms
  of logarithmic schemes}, {Algebra Number Theory}, 10 (2016), pp.~695--735.

\bibitem{Wis16b}
\leavevmode\vrule height 2pt depth -1.6pt width 23pt, {\em Uniqueness of
  minimal morphisms of logarithmic schemes}, arXiv:1601.02968,  (2016).

\bibitem{Yu14a}
{\sc T.~Y. Yu}, {\em Gromov compactness in non-archimedean analytic geometry},
  J. Reine Angew. Math. (Crelle's Journal) arXiv:1401.6452,  (To appear).

\end{thebibliography}

\end{document}